\documentclass[english]{amsart}
\usepackage[T1]{fontenc}
\usepackage[utf8]{inputenc} 

\usepackage{amsmath}
\usepackage{graphicx}
\usepackage{amsthm}
\usepackage{amssymb}
\usepackage{mathtools}
\usepackage{hyperref}

\usepackage{enumitem}
\usepackage{stmaryrd}
\usepackage{xcolor}

\newcommand{\dc}{\#^\mathrm{dc}}
\newcommand{\e}{\varepsilon}

\newcommand{\Lcal}{\mathcal{L}}
\newcommand{\frk}{\mathfrak}
\newcommand{\ol}{\overline}
\newcommand{\Rb}{\mathbb{R}}
\newcommand{\RTr}{\Rb\mathrm{T}_r}
\newcommand{\Lc}{\mathcal{L}}

\providecommand{\dotdiv}{
  \mathbin{
    \vphantom{+}
    \text{
      \mathsurround=0pt 
      \ooalign{
        \noalign{\kern-.35ex}
        \hidewidth$\smash{\cdot}$\hidewidth\cr 
        \noalign{\kern.35ex}
        $-$\cr 
      }%
    }%
  }%
}

\newcommand*{\approxx}{%
  \mathrel{\vcenter{\offinterlineskip
  \hbox{$\sim$}\vskip-.35ex\hbox{$\sim$}\vskip-.35ex\hbox{$\sim$}}}}

\theoremstyle{plain}
\newtheorem{thm}{Theorem}[section]
\newtheorem{prop}[thm]{Proposition}
\newtheorem{lem}[thm]{Lemma}
\newtheorem{cor}[thm]{Corollary}
\newtheorem{fact}[thm]{Fact}

\newtheorem{quest}[thm]{Question}

\theoremstyle{definition}
\newtheorem{defn}[thm]{Definition}
\newtheorem{nota}[thm]{Notation}

\theoremstyle{plain}
\newtheorem{innercustomgeneric}{\customgenericname}
\providecommand{\customgenericname}{}
\newcommand{\newcustomtheorem}[2]{%
  \newenvironment{#1}[1]
  {%
   \renewcommand\customgenericname{#2}%
   \renewcommand\theinnercustomgeneric{##1}%
   \innercustomgeneric
  }
  {\endinnercustomgeneric}
}

\makeatletter
\DeclareRobustCommand{\cset}{\@ifstar\star@cset\normal@cset}
\newcommand{\star@cset}[1]{\left\llbracket#1\right\rrbracket}
\newcommand{\normal@cset}[2][]{\mathopen{#1\llbracket}#2\mathclose{#1\rrbracket}}
\makeatother

\newcustomtheorem{customthm}{Theorem}
\newcustomtheorem{customlemma}{Lemma}
\newcustomtheorem{customprop}{Proposition}

\begin{document}


\title{Approximate Isomorphism of Metric Structures}
\author{James Hanson}
\address{Department of Mathematics\\
  University of Maryland\\
  College Park, MD 20742, USA}
\email{jhanson9@umd.edu}




\begin{abstract}
  We give a formalism for approximate isomorphism in continuous logic simultaneously generalizing those of two papers by Ben Yaacov \cite{OnPert} and by Ben Yaacov, Doucha, Nies, and Tsankov \cite{MSA}, which are largely incompatible. With this we explicitly exhibit Scott sentences for the perturbation systems of the former paper, such as the Banach-Mazur distance and the Lipschitz distance between metric spaces. Our formalism is simultaneously characterized syntactically by a mild generalization of perturbation systems and semantically by certain elementary classes of two-sorted structures that witness approximate isomorphism. As an application, we show that the theory of any $\Rb$-tree or ultrametric space of finite radius is stable, improving a result of Carlisle and Henson~\cite{Carlisle2020}.
\end{abstract}

\maketitle
\vspace{-1em}
\section*{Introduction}

There are many different notions of `approximate isomorphism' in various branches of analysis. The two that are best known are perhaps the Banach-Mazur distance between Banach spaces and the Gromov-Hausdorff distance between metric spaces. These two---as well as their lesser known cousins the Kadets distance between Banach spaces and the Lipschitz distance between metric spaces---seem to have fruitful interaction with continuous logic,\footnote{Some additional, specialized examples, relevant to the model theory of $C^\ast$-algebras, are the completely bounded Banach-Mazur distance in the context of operator spaces and the unnamed distance that induces the weak topology on the class of operator spaces in question \cite{Goldbring2015}.} as explored in \cite{OnPert}, \cite{MSA}, and \cite{NanoThesis} and in \cite{Iovino95stablebanach}\footnote{We should warn the reader that \cite{Iovino95stablebanach}, on the here relevant topic of \emph{uniform structures on type spaces}, points the reader to an in-preparation monograph, \emph{Banach Space Model Theory, I: Basic Concepts and Tools}, that does not seem to exist in any publicly available form.} before the development of the modern formalism of continuous logic. 

In \cite{Iovino95stablebanach}, Iovino discusses the notion of \emph{uniform structures on the space of types} \cite[Def.~2.2]{Iovino95stablebanach} in the context of positive bounded formulas, a seminal precursor formalism to the modern form of continuous logic. In the special case of uniform structures generated by a metric, these amount to the same thing as our induced metrics (Defintion~\ref{defn:induced-metrics-on-type-space}). While it is a fairly routine matter to generalize these ideas to the current formalism of continuous logic, the only published modern account of these ideas is the special case developed in \cite{OnPert}. Aside from committing a fully modern development to print, the primary contributions of this paper are the purely syntactic formalization given in Section~\ref{sec:approx-iso}, the precise semantic, structural characterization shown in Section~\ref{sec:th-approx-iso}, and the infinitary generalizations explored in Section~\ref{sec:Scott-sent}, extending the work of \cite{MSA}. As such, this paper is a synthesis of ideas presented in \cite{OnPert}, \cite{MSA}, and~\cite{Iovino95stablebanach}. 

In \cite{OnPert}, Ben Yaacov introduces perturbation systems---a broad notion of approximate isomorphism---in order to generalize an unpublished result of  Henson's, specifically a Ryll-Nardzewski type characterization of Banach space theories that are `approximately separably categorical' with regards to the Banach-Mazur distance. Ben Yaacov's formalism requires that approximate isomorphisms be witnessed by uniformly continuous bijections with uniformly continuous inverses. As such, while it comfortably covers the Banach-Mazur and Lipschitz distances, it cannot accommodate the Gromov-Hausdorff or Kadets distances. 

In \cite{MSA}, Ben Yaacov, Doucha, Nies, and Tsankov generalize Scott analysis to a family of continuous analogs of $\Lcal_{\omega_1 \omega}$. Among other results, they use this to exhibit Scott sentences that characterize separable metric structures not only up to isomorphism, but also up to Gromov-Hausdorff or Kadets distance 0. Their formalism does not seem to be able to capture the Banach-Mazur or Lipschitz distances, although the existence of Scott sentences capturing these was shown indirectly by their continuous Lopez-Escobar theorem and results in \cite{ComplexDistNew}. In this paper we will show that with a small modification the formalism of \cite{MSA}, we can give Scott sentences for Banach-Mazur and Lipschitz distances, as well as other `well behaved' notions of approximate isomorphism.

All four of the distances mentioned---Banach-Mazur, Kadets, Gromov-Hausdorff, and Lipschitz---can be expressed in terms of `correlations,' i.e.\ total surjective relations between the structures in question (bijections being a special case of correlations) and some notion of `distortion' that measures how good of an approximation of isomorphism the given correlation is. Our formalism will use this as a starting point, defining  distortion in terms of certain appropriate designated collections of formulas, called `distortion systems.' This is a more syntactic way of looking at something very similar to the objects studied in \cite{OnPert}, but without the requirement that the correlations in question be functions. 

After giving three characterizations of our family of notions of approximate isomorphism, we will give an explicit collection of formulas that captures the Banach-Mazur distance, giving a presentation of a result from \cite{Iovino95stablebanach} in a modern formalism. We will extend the result of \cite{MSA} by giving Scott sentences for perturbation systems. And finally, we will use our formalism to show that all rooted $\Rb$-trees of finite radius and all ultrametric spaces of finite radius are stable, improving a result of \cite{Carlisle2020}.

Despite the limitations of \cite{OnPert} and \cite{MSA} it seems that they each settled on fairly natural formalisms. As we will see there are two well-behaved classes of distortion systems. The first includes the Banach-Mazur and Lipschitz distances and is roughly the same thing as the class of perturbation systems. The second includes the Gromov-Hausdorff and Kadets distances and fits fairly well into the weak modulus formalism of \cite{MSA}. We end the paper with a section detailing a simple pathological example of what can happen outside of these two nice classes.

In a follow-up paper we will generalize the approximate Ryll-Nardzewski theorem of \cite{OnPert} to the context of certain well behaved distortion systems (including in particular the Gromov-Hausdorff and Kadets distances), as well as give a partial Morley's theorem for `inseparably approximately categorical' theories.

For the general formalism of continuous logic and the majority of the notation used here, see \cite{MTFMS}. As opposed to \cite{MTFMS}, however, we (implicitly) opt for the extended definition of formula they allude to in and before Proposition 9.3. Specifically, we allow arbitrary continuous functions from $\mathbb{R}^\omega$ to $\mathbb{R}$ as connectives. This means that formulas may use countably many constants and have countably many free variables, but also, most importantly, that the collection of formulas is closed under uniformly convergent limits up to logical equivalence.  Note that this does not increase the expressiveness of first-order continuous logic, despite the presence of infinitary connectives. In continuous generalizations of $\Lcal_{\omega_1\omega}$, the expressiveness is fundamentally increased in that the infinitary connectives introduced there ($\sup$ and $\inf$ of sequences of formulas) are not continuous.

 Here are the rest of the notational conventions used in this paper.







\begin{nota}
\hfill
\begin{itemize}
\item For any metric space $(X,d)$ and any $A\subseteq X$, we will write $\dc (A,d)$ for the metric density character of $A$ with regards to $d$, i.e., the minimum cardinality of a $d$-dense subset of $A$. We will drop the $d$ if the metric is clear from context.

\item To avoid confusion with the established logical roles of $\wedge$ and $\vee$ we will avoid using this notation for $\min$ and $\max$ but in the interest of conciseness we will let $x \uparrow y \coloneqq \max(x,y)$ and $x \downarrow y \coloneqq \min(x,y)$. Note that  $\mathfrak{M} \models \varphi \uparrow \psi \leq 0$ if and only if $\mathfrak{M} \models \varphi \leq 0$ and $\mathfrak{M} \models \psi \leq 0$, and likewise $\mathfrak{M} \models \varphi \downarrow \psi \leq 0$ if and only if $\mathfrak{M} \models \varphi \leq 0$ or $\mathfrak{M} \models \psi \leq 0$.

We take $\uparrow$ and $\downarrow$ to have higher binding precedence than addition but lower binding precedence than multiplication, so for example $ab\uparrow c+d = ((ab)\uparrow c)+d$. We will never write expressions like $x\uparrow a \downarrow b$.

\item We will write $[x]_a^b$ for $(x\uparrow a)\downarrow b$. Note that $[x]_a^b = x$ for $x \in [a,b]$, $[x]_a^b = a$ if $x\leq a$, and $[x]_a^b = b$ if $x \geq b$.

\item If $\varphi$ is a formula and $r$ is a real number (or perhaps another formula), we will write expressions such as $\cset{\varphi < r}$ and $\cset{\varphi \geq r}$ to represent the sets of types (in some type space that will be clear from context) satisfying the given condition.
\end{itemize}
\end{nota}

\section{Approximate Isomorphism}
\label{sec:approx-iso}

\begin{defn}
Fix a language $\Lcal$ with sorts $\mathcal{S}$, $\Lcal$-pre-structures $\frk{M}$ and $\frk{N}$, and tuples $\bar{m} \in \frk{M}$ and $\bar{n} \in \frk{N}$ of the same length with elements in the same sorts.
\begin{itemize}
    \item[(i)]  The \emph{sort-by-sort product of $\frk{M}$ and $\frk{N}$}, written $\frk{M}\times_\mathcal{S}\frk{N}$, is the collection $\bigsqcup_{s\in\mathcal{S}} s(\frk{M})\times s(\frk{N})$. If $\Lcal$ is single-sorted we will take $\times_\mathcal{S}$ to be the ordinary Cartesian product.
    \item[(ii)] A \emph{correlation between $\frk{M}$ and $\frk{N}$} is a set $R \subseteq \frk{M} \times_\mathcal{S} \frk{N}$ such that for each sort $s$,  $R\upharpoonright s \coloneqq R\upharpoonright s(\frk{M})\times s(\frk{N})$ is a total surjective relation. We will write $\mathrm{cor}(\frk{M},\bar{m};\frk{N},\bar{n})$ for the collection of correlations between $\frk{M}$ and $\frk{N}$ such that for each index $i$ less than the length of $\bar{m}$, $(m_i,n_i)\in R$ (for any binary relation we will abbreviate this condition as $(\bar{m},\bar{n})\in R$). If $\bar{m}$ and $\bar{n}$ are empty we will write $\mathrm{cor}(\frk{M},\frk{N})$. 
    \item[(iii)] An \emph{almost correlation between $\frk{M}$ and $\frk{N}$} is a correlation between dense sub-pre-structures of $\frk{M}$ and $\frk{N}$. We will write $\mathrm{acor}(\frk{M},\bar{n};\frk{N},\bar{m})$ for the collection of almost correlations $R$ between $\frk{M}$ and $\frk{N}$ such that $(\bar{m},\bar{n})\in R$.
    
\end{itemize}
\end{defn}
Note that there is no requirement that correlations or almost correlations are closed. This will turn out to be inessential, but it is convenient for constructing them.

Almost correlations are natural to consider for two reasons. If $(\frk{M},\frk{N},R)$ is a metric structure in which $R$ is a definable subset of $\frk{M}\times\frk{N}$, then there is a sentence that holds if and only if $R$ is an almost correlation. There is no sentence that holds precisely when $R$ is a correlation. The other reason is that many constructions, such as back-and-forth constructions, naturally build an enumeration of a dense sub-pre-structure rather than an enumeration of the entire structure in question. This means that when one tries to build a correlation with some kind of iterative construction, one will often only build an almost correlation.


\begin{defn}
Let $\Delta$ be a collection of (finitary) $\Lcal$-formulas and let $T$ be an $\Lcal$-theory. Let $\frk{M},\frk{N} \models T$ with $\bar{m}\in\frk{M}$ and $\bar{n}\in\frk{N}$.
\begin{enumerate}[label=(\roman*)]
    \item For any relation $R \subseteq \frk{M}\times_\mathcal{S} \frk{N}$, we define the \emph{distortion of $R$ with respect to $\Delta$} as follows:
$$\mathrm{dis}_\Delta(R) =\sup\{|\varphi^\frk{M}(\bar{m})-\varphi^\frk{N}(\bar{n})|:\varphi \in \Delta, (\bar{m},\bar{n})\in R\}$$
    \item We define the \emph{$\Delta$-distance between $(\frk{M},\bar{m})$ and $(\frk{N},\bar{n})$} as follows:
    $$\rho_\Delta(\frk{M},\bar{m};\frk{N},\bar{n})=\inf \{\mathrm{dis}_\Delta(R) :R\in \mathrm{cor}(\frk{M},\bar{m};\frk{N},\bar{n})\}$$
    If $\bar{m}$ and $\bar{n}$ are empty we will just write $\rho_\Delta(\frk{M},\frk{N})$.
    \item We define the \emph{almost $\Delta$-similarity between $(\frk{M},\bar{m})$ and $(\frk{N},\bar{n})$} as follows:
    $$a_\Delta(\frk{M},\bar{m};\frk{N},\bar{n})=\inf \{\mathrm{dis}_\Delta(R) :R\in \mathrm{acor}(\frk{M},\bar{m};\frk{N},\bar{n})\}$$
    If $\bar{m}$ and $\bar{n}$ are empty we will just write $a_\Delta(\frk{M},\frk{N})$.
    \item We say that $(\frk{M},\bar{m})$ and $(\frk{N},\bar{n})$ are \emph{$\Delta$-approximately isomorphic}, written $(\frk{M},\bar{m})\approxx_\Delta (\frk{N},\bar{n})$, if $\rho_\Delta(\frk{M},\bar{m};\frk{N},\bar{n})=0$. 
    \item We say that $(\frk{M},\bar{m})$ and $(\frk{N},\bar{n})$ are \emph{almost $\Delta$-approximately isomorphic} if $a_\Delta(\frk{M},\bar{m};\allowbreak\frk{N},\bar{n})=0$.
\end{enumerate}
\end{defn}

Note that $a_\Delta(\frk{M},\bar{m};\frk{N},\bar{n})=0$ does not necessarily imply that there are dense sub-pre-structures $\frk{M}_0 \subseteq \frk{M}$ and $\frk{N}_0 \subseteq \frk{N}$ such that $\frk{M}_0 \approxx_\Delta \frk{N}_0$.

Given that the composition of correlations is a correlation, it is very easy to verify that $\rho_\Delta$ is a pseudo-metric and that $\approxx_\Delta$ is an equivalence relation on the class of models of $T$.
$a_\Delta$ in general is somewhat pathological. It can fail the triangle inequality and therefore in particular be different from $\rho_\Delta$. It can even occur that $\rho_\Delta(\frk{M},\frk{N})=\infty$ while $a_\Delta(\frk{M},\frk{N})=0$. We will examine an example of this in Section \ref{sec:Irreg}. Since it is likely (although currently unknown) that almost approximate isomorphism is not transitive, we won't give it a symbol that suggests an equivalence relation. These difficulties only occur with unfamiliar notions of approximate isomorphism. Later on we will identify two common conditions that each ensure $\rho_\Delta = a_\Delta$.

Here are some familiar examples.

\begin{itemize}
    \item If $\Delta$ is the collection of all $\Lcal$-formulas, then $\rho_\Delta(\frk{M},\frk{N}) < \infty$ if and only if $\frk{M} \cong \frk{N}$.
    \item If $T$ is the empty theory in the empty signature and $\Delta = \{\frac{1}{2}d(x,y)\}$, then $\rho_\Delta=d_{\mathrm{GH}}$, the Gromov-Hausdorff distance.
    \item If $T$ is the theory of (the unit balls of) Banach spaces, and $\Delta$ is the collection of all formulas of the form $\left\lVert \sum_i \lambda_i x_i\right\lVert$, with $\sum_i \left|\lambda_i\right| \leq 1$, then $\rho_\Delta=d_K$, the Kadets distance between $\frk{M}$ and $\frk{N}$.
    \item If $T$ is the empty theory in the empty signature and $\Delta$ is the collection of all formulas of the form $[\log d(x,y)]_{-r}^r$, then $\rho_\Delta=d_\mathrm{Lip}$, the Lipschitz distance.
    \item If $\Delta$ is the collection of formulas that are $1$-Lipschitz in any model of $T$, then $$\rho_\Delta(\frk{M},\frk{N})=\inf\{d_H^\frk{C}(f(\frk{M}),g(\frk{N}))|f:\frk{M}\preceq \frk{C},g:\frk{N}\preceq\frk{C}\},$$
    where $d_H$ is the Hausdorff metric. This is a sort of elementary embedding variant of the Gromov-Hausdorff and Kadets distances.
\end{itemize}

The Banach-Mazur distance can also be formalized in this way but the specification of $\Delta$ is somewhat more complicated, so we will leave it to later.

\begin{prop} Fix $\Delta$, a collection of formulas.
\begin{itemize}
    \item[(i)] For any relation $R\subseteq \frk{M}\times_\mathcal{S} \frk{N}$, $\mathrm{dis}_\Delta (R) =\mathrm{dis}_\Delta (\ol{R})$, where $\ol{R}$ is the metric closure of $R$ (in each sort).
    \item[(ii)] There is a subset $\Delta_0 \subseteq \Delta$ with $|\Delta_0| = |\Lcal|$, such that $\mathrm{dis}_\Delta = \mathrm{dis}_{\Delta_0}$.
\end{itemize}
\end{prop}

\begin{proof}
\emph{(i):} This follows from the uniform continuity of each $\varphi \in \Delta$.

\emph{(ii):} Choose a subset which is dense in $\Delta$ with regards to uniform convergence. Since the density character of the space of all $\Lcal$-formulas under uniform convergence is $|\Lcal|$, it is always possible to find such a $\Delta_0$.
\end{proof}

 Given a collection $\Delta$, we can often enlarge it in a natural way without changing the value of $\rho_\Delta$.

\begin{defn}
If $\Delta$ is a collection of formulas, let $\ol{\Delta}$ be the closure of $\Delta$ under renaming variables, quantification, $1$-Lipschitz connectives, logical equivalence modulo $T$, and uniformly convergent limits.
\end{defn}

Recall that $1$-Lipschitz is meant in the sense of the max metric, i.e.\ a connective $F(\bar{x})$ is $1$-Lipschitz if $|x_i - y_i|\leq \e$ for all $i$ implies $|F(\bar{x})-F(\bar{y})|\leq \e$.

\begin{prop} \label{prop:enlarge}
For any $\frk{M},\frk{N}\models T$ and $\bar{m}\in\frk{M}$, $\bar{n}\in\frk{N}$, $\mathrm{dis}_\Delta(R) = \mathrm{dis}_{\ol{\Delta}}(R)$, so in particular $\rho_\Delta(\frk{M},\bar{m};\frk{N},\bar{n}) = \rho_{\ol{\Delta}}(\frk{M},\bar{m};\frk{N},\bar{n})$.
\end{prop}

\begin{proof}
It is clear that if $\Delta$ is any set of formulas and $\Sigma$ is $\Delta$ closed under renaming variables, $1$-Lipschitz connectives, logical equivalence modulo $T$, and uniformly convergent limits, then $\mathrm{dis}_\Sigma \leq \mathrm{dis}_\Delta$. Since $\Sigma\supseteq \Delta$, we also clearly have $\mathrm{dis}_\Delta \leq \mathrm{dis}_\Sigma$, so $\mathrm{dis}_\Sigma =\mathrm{dis}_\Delta$.

Furthermore if $\{\Sigma_i\}_{i<\lambda}$ is some increasing chain of sets of formulas such that $\mathrm{dis}_\Delta=\mathrm{dis}_{\Sigma_i}$ for every $i<\lambda$, then we also have that $\mathrm{dis}_\Delta = \mathrm{dis}_{\Sigma_\lambda}$, where $\Sigma_\lambda =\bigcup_{i<\lambda} \Sigma_i$.

So the only difficulty is showing that quantification is safe. Suppose that $\varphi(\bar{x},y) \in \Sigma$. We want to show that for any structures $\frk{M},\frk{N}\models T$, $\bar{m}\in\frk{M}$,  $\bar{n}\in\frk{N}$, and $R\in\mathrm{cor}(\frk{M},\bar{m};\frk{N},\bar{n})$,  $|\inf_y\varphi^\frk{M}(\bar{m},y)-\inf_y \varphi^\frk{N}(\bar{n},y)| \leq \mathrm{dis}_{\Sigma}(R)$.

For each $\e>0$, find an $a \in \frk{M}$ such that $\varphi^\frk{M}(\bar{m},a) < \inf_y \varphi^\frk{M}(\bar{m},y)+\e$. Since $R$ is a correlation we can find a $b\in \frk{N}$ such that $(a,b)\in R$, so we must have that 
\begin{align*}
\inf_y \varphi^\frk{N}(\bar{n},y) &\leq \varphi^\frk{N}(\bar{n},b) <  \varphi^\frk{M}(\bar{m},a) + \mathrm{dis}_{\Sigma}(R), \\
\inf_y \varphi^\frk{N}(\bar{n},y) &< \inf_y \varphi^\frk{M}(\bar{m},a) + \mathrm{dis}_{\Sigma}(R) + \e.
\end{align*}
By symmetry we have that $|\inf_y \varphi^\frk{M}(\bar{m},y) - \inf_y \varphi^\frk{N}(\bar{n},y)| \leq \mathrm{dis}_{\Sigma}(R)+ \e$,
and since we can do this for any $\e > 0$ we have that $|\inf_y\varphi^\frk{M}(\bar{m},y)-\inf_y \varphi^\frk{N}(\bar{n},y)| \leq \mathrm{dis}_{\Sigma}(R)$ as required. Since $x\mapsto -x$ is a $1$-Lipschitz connective, we have this for $\sup$ as well.

So by iteratively alternating between closing under renaming variables, $1$-Lipschitz connectives, logical equivalence modulo $T$, and uniformly convergent limits on the one hand and closure under connectives on the other, we can form a chain $\{\Delta_i\}_{i<\omega_1}$ whose union is $\ol{\Delta}$ and which has the property that for every $i<\omega_1$, $\mathrm{dis}_{\Delta_i}(R)=\mathrm{dis}_\Delta(R)$. Therefore we have that $\mathrm{dis}_\Delta(R) = \mathrm{dis}_{\ol{\Delta}}(R)$, as required.
\end{proof}

If we require one more thing from $\Delta$ we get something more, and we can justify the name `approximate isomorphism' in that it is, naturally, approximately isomorphism:

\begin{prop} Let $\Delta=\ol{\Delta}$ be logically complete. If $(\frk{M},\bar{m})$ and $(\frk{N},\bar{n})$ are almost $\Delta$-approximately isomorphic (so in particular if $(\frk{M},\bar{m})\approxx_\Delta (\frk{N},\bar{n})$), then for any non-principal ultrafilter $\mathcal{U}$ on $\omega$, $(\frk{M},\bar{m})^\mathcal{U} \cong (\frk{N},\bar{n})^\mathcal{U}$, and therefore  $(\frk{M},\bar{m})\equiv(\frk{N},\bar{n})$.

\end{prop}

\begin{proof}
 Let $\{R_i\}_{i<\omega}$ be a sequence of closed almost correlations between $\frk{M}$ and $\frk{N}$ such that $(\bar{m},\bar{n})\in R_i$ and $\mathrm{dis}_\Delta (R_i) \leq 2^{-i}$. Let $\mathcal{U}$ be a non-principal ultrafilter on $\omega$. For each $i$, let $(\frk{M},\frk{N},R_i)$ be a metric structure containing $\frk{M}$ and $\frk{N}$ in different sorts and having distance predicates for the sets $R_i\upharpoonright s \subseteq s(\frk{M})\times s(\frk{N})$ for each $s\in\mathcal{S}$. Consider the structure $(\frk{M}^\prime,\frk{N}^\prime,R^\prime) = \prod_{i<\omega} (\frk{M},\frk{N},R_i)/\mathcal{U}$. Clearly $\frk{M}^\prime$ and $\frk{N}^\prime$ are elementary extensions of $\frk{M}$ and $\frk{N}$, respectively. Furthermore by $\aleph_1$-saturation (in a countable reduct if the language is uncountable) we have that $R^\prime$ is a correlation, rather than just an almost correlation.
 
 Finally for each formula $\varphi \in \Delta$, we have that $$(\frk{M},\frk{N},R_i)\models \sup_{(\bar{x},\bar{y}) \in R_i} |\varphi^\frk{M}(\bar{x},\bar{m})-\varphi^\frk{N}(\bar{y},\bar{n})| \leq 2^{-i}.$$  This is expressible because $R_i$ is a definable set. Therefore we have that $$(\frk{M}^\prime,\frk{N}^\prime,R^\prime)\models \sup_{(\bar{x},\bar{y}) \in R_i} |\varphi^\frk{M}(\bar{x},\bar{m}^\prime)-\varphi^\frk{N}(\bar{y},\bar{n}^\prime)| \leq 0$$ for each formula $\varphi \in \Delta$. By the logical completeness of $\Delta$, this implies that $R^\prime$ is the graph of an isomorphism between $\frk{M}^\prime$ and $\frk{N}^\prime$ with $(\bar{m}^\prime,\bar{n}^\prime)\in R^\prime$.
\end{proof}

So we will give $\Delta$'s with these properties a name.

\begin{defn}
A set of formulas $\Delta$ is a \emph{distortion system for $T$} if it is logically complete and closed under renaming variables, quantification, $1$-Lipschitz connectives, logical equivalence modulo $T$, and uniformly convergent limits.
\end{defn}

In many of the motivating examples we aren't given a $\Delta$ that is already a distortion system and it's not immediately clear whether or not $\ol{\Delta}$ will be logically complete. There is an easy test, however.

\begin{defn}
A collection of $\Lcal$-formulas $\Delta$ is $\emph{atomically complete}$ if, after closing under renaming of variables, any quantifier free type $p$ is entirely determined by the values of $\varphi(p)$ for $\varphi \in \Delta$.
\end{defn}

\begin{prop} \label{prop:atom-comp}
If $\Delta$ is atomically complete, then $\ol{\Delta}$ is a distortion system.
\end{prop}
\begin{proof} Clearly we only need to show that $\ol{\Delta}$ is logically complete. Let $\varphi$ be an atomic formula and let $r<s$ be real numbers.

\emph{Claim:} There is a $\ol{\Delta}$-formula $\psi$ and real numbers $u < v$ such that $\varphi\leq r \vdash \psi < u$ and $\varphi \geq s \vdash \psi > v$.

\emph{Proof of claim:} Let $p$ be a quantifier free type with $\varphi(p)\leq r$. By compactness there must exist a finite list $\chi_1,\dots,\chi_k$ of $\Delta$-formulas and an $\e_p > 0$ such that $\cset{|\chi_1-\chi_1(p)|\leq \e_p}\cap\dots \cap \cset{|\chi_k-\chi_k(p)|\leq \e_p}$ is disjoint from $\cset{\varphi \geq s}$. In particular $\xi_p=|\chi_1 -\chi_1(p)| \uparrow\dots \uparrow |\chi_k-\chi_k(p)|$ is a $\ol{\Delta}$-formula. By compactness there is a finite list $p_1,\dots,p_\ell$ of quantifier free types such that $\cset{\xi_{p_1} < \e_{p_1}} \cup \dots \cup \cset{\xi_{p_\ell} < \e_{p_\ell}}$ covers $\cset{\varphi \leq r}$. Furthermore we still have that $\cset{\xi_{p_1} \leq \e_{p_1}} \cup \dots \cup\cset{\xi_{p_\ell} \leq \e_{p_\ell}}$ is disjoint from $\cset{\varphi \geq s}$. Let $\theta = (\xi_{p_1} - \e_{p_1}) \downarrow \dots \downarrow (\xi_{p_\ell} - \e_{p_\ell})$, and note that this a $\ol{\Delta}$-formula. We have that $\cset{\varphi \leq r} \subseteq \cset{\theta < 0}$, and by compactness there must be some $\delta > 0$ such that $\cset{\varphi \geq s} \subseteq \cset{\theta >\delta}$, as required. \hfill $\qed_{\text{claim}}$ 

Let $\Sigma$ be the set of formulas $\varphi$ that are in prenex form with quantifier free part $\psi$ such that $\psi$ is a maximum of minimums of formulas of the form $\alpha - r$, $r - \alpha$, or $r$ where $\alpha$ is an atomic formula and $r$ a real number. $\Sigma$ is logically complete in any signature.

Fix a type $p \in S_n(T)$ and consider a formula $\varphi \in \Sigma$ such that $p\vdash \varphi \leq 0$. Fix $\e > 0$ and let the quantifier free part of $\varphi$ be $\psi=\max_{i<k}\min_{j<\ell(i)}\chi_{i,j}$. Define formulas $\chi^\prime_{i,j}$ like so: 
\begin{itemize}
    \item If $\chi_{i,j} = r$, then $\chi^\prime_{i,j} = r$ and $\delta_{i,j} = 1$.
    \item If $\chi_{i,j} = \alpha - r$, find a $\ol{\Delta}$-formula $\eta$ and real numbers $u<v$ such that $\cset{\alpha \leq r} \subseteq \cset{\eta < u}$ and $\cset{\alpha \geq r+\e}\subseteq \cset{\eta > v}$, and set $\chi_{i,j}^\prime = \eta - u$ and let $\delta_{i,j} = \frac{v-u}{2}$.
    \item If $\chi_{i,j} = r - \alpha$, find a $\ol{\Delta}$-formula $\eta$ and real numbers $u<v$ such that $\cset{\alpha \leq r - \e} \subseteq \cset{\eta < u}$ and $\cset{\alpha \geq r}\subseteq \cset{\eta > v}$, and set $\chi_{i,j}^\prime = v - \eta$ and let $\delta_{i,j} = \frac{v-u}{2}$.
\end{itemize}
Set $\delta = \min_{i,j}\delta_{i,j}$ and $\psi^\prime = \max_{i<k}\min_{j<\ell(i)}\chi_{i,j}^\prime$. Then let $\varphi^\prime$ be $\psi^\prime$ with the same quantifiers that $\varphi$ has. Now we have constructed a $\ol{\Delta}$-formula, $\varphi^\prime$ such that for any type $q$, if $q \vdash \varphi \leq 0$, then $q\vdash \varphi^\prime\leq 0$ and if $q\vdash \varphi^\prime \leq \delta$, then $q\vdash \varphi \leq \e$. Since we can do this for any $\varphi \in p$ and any $\e > 0$, we have that $p$ is entirely determined by $\{\varphi : p \vdash \varphi \in \ol{\Delta}\}$, as required.
\end{proof}


Now is a convenient time to introduce the following notions:

\begin{defn} Let $\Delta$ be a distortion system for $T$.
\begin{itemize}
    \item[(i)] We say that $\Delta$ is \emph{regular} if there is an  $\e>0$ such that for any models $\mathfrak{M},\mathfrak{N}\models T$, any almost correlation $R \in \mathrm{acor}(\frk{M},\frk{N})$ with $\mathrm{dis}_\Delta(R) < \e$, and any $\delta> 0$, there exists a correlation $S\in\mathrm{cor}(\frk{M},\frk{N})$ such that $S \supseteq R$ and $\mathrm{dis}_\Delta(S) \leq \mathrm{dis}_\Delta(R) + \delta$.
    \item[(ii)] We say that $\Delta$ is \emph{functional} if there is an $\e > 0$ such that for any models $\mathfrak{M},\mathfrak{N}\models T$ and any closed $R \in \mathrm{acor}(\frk{M},\frk{N})$, if $\mathrm{dis}_\Delta(R) < \e$, then $R$ is the graph of a uniformly continuous bijection between $\frk{M}$ and $\frk{N}$ with uniformly continuous inverse. 
    \item[(iii)] We say that $\Delta$ is \emph{uniformly uniformly continuous} or \emph{u.u.c.}\ if for every $\e>0$ there exists a $\delta>0$ such that for any models $\mathfrak{M},\mathfrak{N}\models T$ and any almost correlation $R \in \mathrm{acor}(\frk{M},\frk{N})$, $\mathrm{dis}_\Delta(R^{< \delta}) \leq \mathrm{dis}_\Delta(R) + \e$, where $R^{< \delta} = \{(a,b) : (\exists(c,d) \in R) d^{\frk{M}}(a,c),d^{\frk{N}}(b,d) < \delta\}$. \qedhere
\end{itemize}
\end{defn}

Obviously functional and u.u.c.\ distortion systems are regular. It is easy to construct regular distortion systems that are neither by `gluing' together functional and u.u.c.\ distortion systems, such as a two-sorted theory in which both sorts are metric spaces and we simultaneously consider Gromov-Hausdorff distance on the first sort and Lipschitz distance on the second sort.

Functional distortion systems are essentially the same as Ben Yaacov's perturbations. The Gromov-Hausdorff and Kadets distances are u.u.c., and u.u.c.\ distortion systems are natural generalizations of the Gromov-Hausdoff and Kadets distances. Furthermore one can show that in some common cases the back-and-forth metrics of \cite{MSA} must be either isomorphism itself or be equivalent to $\rho_\Delta$ for $\Delta$, a u.u.c.\ distortion system. 

\begin{prop} \label{prop:nice-regular} Let $\Delta$ be a distortion system.
\begin{itemize}
    \item[(i)] If $\Delta$ is regular, then for any $(\frk{M},\bar{m})$ and $(\frk{N},\bar{n})$, $$\rho_\Delta(\frk{M},\bar{m};\frk{N},\bar{n}) = a_\Delta(\frk{M},\bar{m};\frk{N},\bar{n}).$$ 
    \item[(ii)] $\Delta$ is functional if and only if there is an $\e>0$ such that for any $\delta > 0$ there is a formula $\varphi(x,y)\in \Delta$ such that for any $\frk{M}\models T$ and  $a,b\in \frk{M}$, $\varphi^\frk{M}(a,a)=0$ and if $\varphi^\frk{M}(a,b) < \e$, then $d^\frk{M}(a,b)<\delta$.
    \item[(iii)] $\Delta$ is u.u.c.\ if and only if it is uniformly uniformly continuous as a set of formulas, i.e.\ there is a single modulus $\alpha:\mathbb{R}\rightarrow \mathbb{R}$ (continuous and with $\alpha(0)=0$) such that for any $\bar{a},\bar{b} \in \frk{M}\models T$, $|\varphi^\frk{M}(\bar{a})-\varphi^\frk{M}(\bar{b})| \leq \alpha(d^\frk{M}(\bar{a},\bar{b}))$.
\end{itemize}
\end{prop}
\begin{proof}
\emph{(i):} Given any almost correlations witnessing the value of $a_\Delta$, regularity immediately gives us correlations witnessing the same value for $\rho_\Delta$.

\emph{(ii):} We will defer the proof of this until later (also labeled Proposition~\hyperlink{repeated}{\ref*{prop:nice-regular}} after Fact~\ref{fact:Approx_Iso_in_Cts_Logic:1}) when machinery is available to make the proof easier.

\emph{(iii):} The $(\Rightarrow)$ direction follows easily from considering the identity as a correlation on models of $T$. The $(\Leftarrow)$ direction is obvious.
\end{proof}

Cauchy sequences in $\rho_\Delta$ give us a way of constructing a limiting structure.

\begin{lem} \label{lem:coarsest}
Let $\Delta$ be a distortion system. For every predicate symbol $P$ and every $\e > 0$ there is a $\delta > 0$ such that if $\rho_\Delta(\frk{M},\bar{m};\frk{N},\bar{n}) < \delta$ then $|P^\frk{M}(\bar{m}) - P^\frk{N}(\bar{n})| < \e$.
\end{lem}

\begin{proof}
This follows from the fact that $\Delta$ is logically complete and compactness.
\end{proof}

\begin{prop}
Let $\Delta$ be a distortion system for $T$.

\begin{itemize}
    \item[(i)] If $\{\frk{M}_i,\bar{m}_i\}_{i<\omega}$ is a sequence of pre-models of $T$ such that for each $i<\omega$, $$\rho_\Delta(\frk{M}_i,\bar{m}_i;\frk{M}_{i+1},\bar{m}_i) < 2^{-i},$$ then there is a pre-structure $\frk{M}_\omega \models T$ with $\bar{m}_\omega$ such that for each $i<\omega$, $$\rho_\Delta(\frk{M}_i,\bar{m}_i;\frk{M}_\omega,\bar{m}_\omega) \leq 2^{-i+1}.$$

Furthermore $\dc \frk{M}_\omega \leq \sup_i \dc \frk{M}_i$, and if all the $\frk{M}_i$ are metrically compact then $\frk{M}_\omega$ is metrically compact.
    \item[(ii)] If $\Delta$ is regular and the $\frk{M}_i$ are complete structures, then $\frk{M}_\omega$ can be taken to be a complete structure.
\end{itemize}

\end{prop}
\begin{proof}
\emph{(i):} For each $i<\omega$ find closed $R_{i+\frac{1}{2}} \in \mathrm{cor}(\frk{M}_i,\bar{m}_i;\frk{M}_{i+1},\bar{m}_{i+1})$ such that $\mathrm{dis}_\Delta(R_{i+\frac{1}{2}}) < 2^{-i}$.

Let $M_\omega^0 = \{x \in \prod_{i}M_i : (\forall i) (x(i),x(i+1)) \in R_{i+\frac{1}{2}}\}$.  For any $\bar{a} \in M_\omega^0$ and predicate symbol $P$, define $P^{\frk{M}_\omega^0}(\bar{a})$ to be $\lim_{i\rightarrow \infty} P^{\frk{M}_i}(\bar{a}(i))$. By Lemma \ref{lem:coarsest} this limit always exists. Furthermore we have that $d^{\frk{M}_\omega}$ is a pseudo-metric on $M_\omega^0$ and that all predicate symbols $P$ obey the correct moduli of continuity for the signature $\Lcal$. So let $\frk{M}_\omega$ be $M_\omega^0$ modded by $d^{\frk{M}_\omega} = 0$, and we have that this is an $\Lcal$-pre-structure.

Now to see that $\frk{M}_\omega\models T$, we will show that for any restricted formula $\varphi(\bar{x})$, $\varphi^{\frk{M}_\omega}(\bar{a}) = \lim_{i\rightarrow \infty} \varphi^{\frk{M}_i}(\bar{a}(i))$ and furthermore that this convergence is uniform in $\bar{a}$. We already have that this is true for atomic formulas, and if $F$ is a connective and we've shown that this holds for some tuple $\bar{\varphi}$ of formulas, then it clearly holds for $F(\bar{\varphi})$ as well. So all that we need to do is show that this is true for quantification. Let $\varphi(\bar{x},y)$ be a formula for which $\varphi^{\frk{M}_\omega}(\bar{a},b) = \lim_{i\rightarrow \infty} \varphi^{\frk{M}_i}(\bar{a}(i),b(i))$ uniformly in $\bar{a}b$. Fix $\e > 0$ and find a $j<\omega$ large enough that $|\varphi^{\frk{M}_\omega}(\bar{a},b) - \varphi^{\frk{M}_i}(\bar{a}(i),b(i))| < \frac{1}{2}\e$ for all $\bar{a}b$ and $i\geq j$. Now find $c \in \frk{M}_j$ such that $\varphi^{\frk{M}_j}(\bar{a}(j),c) < \inf_y \varphi^{\frk{M}_j}(\bar{a}(j),y) + \frac{1}{2}\e$. Extend $c$ to a sequence $e(i) \in M_\omega^0$ such that $e(j)=c$. Now we have that $$\varphi^{\frk{M}_\omega}(\bar{a},e)  < \varphi^{\frk{M}_j}(\bar{a}(j),e(j)) + \frac{1}{2}\e,$$ so in particular 
$$\inf_y \varphi^{\frk{M}_\omega}(\bar{a},y) \leq\varphi^{\frk{M}_\omega}(\bar{a},e)  < \inf_y \varphi^{\frk{M}_j}(\bar{a}(j),y) + \e.$$
Since we can do this for $\varphi$ and $-\varphi$ and for any $\e > 0$, we have shown the required property for $\inf_y \varphi(\bar{x},y)$. Therefore, by induction, this holds for all restricted formulas and thus, by uniform convergence, for all formulas.

Since sentences are formulas we have that for any $\varphi$ such that $\frk{M}_i\models \varphi\leq 0$ for all $i<\omega$, $\frk{M}_\omega \models \varphi \leq 0$ as well.

To show the bound on the density character of $\frk{M}_\omega$, assume that $\dc \frk{M}_i \leq \kappa$ (for some infinite $\kappa$) for each $i<\omega$, and for each such $i$ find a dense subset $X_i\subseteq \frk{M}_i$ of cardinality $\leq \kappa$. For each $a \in X_i$, choose some $b_a \in M_\omega^0$ such that $b_a(i) = a$ and let $X = \{b_a:(\exists i)a\in X_i\}$. Since $d$ uniformly converges this clearly is a dense subset of $\frk{M}_\omega$ as well.

For the statement regarding compact structures, Lemma \ref{lem:coarsest} implies that the sequence of underlying metric spaces of the $\frk{M}_i$ are converging in the Gromov-Hausdorff metric to the underlying metric space of $\frk{M}_\omega$. It is well known that a sequence of compact metric spaces converging in the Gromov-Hausdorff metric converges to a compact metric space, so the result follows.
  
\emph{(ii):} This follows easily from the fact that the correlations between the $\frk{M}_i$ and $\frk{M}_\omega$ are almost correlations between the $\frk{M}_i$ and the completion of $\frk{M}_\omega$.
\end{proof}

\begin{cor}
Let $\Delta$ be a distortion system for $T$. Let $\mathrm{PreMod}(T,\leq\kappa)$ be the collection of pre-models of $T$ with density character $\leq \kappa$,  $\mathrm{Mod}(T,\leq\kappa)$ be the collection of models of $T$ with density character $\leq \kappa$, and let $\mathrm{Mod}(T,\leq \omega^-)$ be the collection of compact models of $T$. For every $\kappa$,
\begin{itemize}
    \item[(i)] $(\mathrm{PreMod}(T,\leq\kappa),\rho_\Delta)$ is a complete pseudo-metric space.
    \item[(ii)] If $\Delta$ is regular then $(\mathrm{Mod}(T,\leq\kappa),\rho_\Delta)$ is a complete pseudo-metric space.
    \item[(iii)] $(\mathrm{Mod}(T,\leq \omega^-),\rho_\Delta)$ is a complete metric space. Furthermore, for compact models, $\rho_\Delta = a_\Delta$.
\end{itemize}
\end{cor}

\begin{proof}
\emph{(i)} and \emph{(ii)} are obvious from the previous proposition. 

\emph{(iii):} The furthermore statement follows from the fact that almost correlations between compact structures are actually correlations, by compactness.

The furthermore statement in part \emph{(i)} of the previous proposition implies that $(\mathrm{Mod}(T,\leq \omega^-),\rho_\Delta)$ is complete, so we just need to show that for compact structures, $\frk{M} \approxx_\Delta \frk{N}$ if and only if $\frk{M} \cong \frk{N}$. But this is easy: Take an ultraproduct of the structures $(\frk{M},\frk{N},R_i)$ where $R_i$ is the correlation taken as a definable subset of $\frk{M}$ and $\frk{N}$ with $\mathrm{dis}_\Delta(R_i) \leq 2^{-i}$. Then you will get a structure of the form $(\frk{M},\frk{N},R_\omega)$ with $R_\omega$ the graph of an isomorphism.
\end{proof}

There is an example of an irregular distortion system $\Delta$ for a theory $T$ such that $(\mathrm{Mod}(T,\leq\kappa),\rho_\Delta)$ is not complete (see Section \ref{sec:Irreg}).

\subsection{Induced Metrics on Type Space}

Any distortion system $\Delta$ for some theory $T$ naturally induces a family of topometrics on the type spaces of $T$. We will define this for one-sorted theories for readability, but the extension to many sorted theories is obvious.

\begin{defn}\label{defn:induced-metrics-on-type-space}
Let $\Delta$ be a distortion system for $T$. For each $\lambda$ and any $p,q\in S_\lambda$, let
$$\delta_\Delta^\lambda(p,q) = \inf\{\rho_\Delta(\frk{M},\bar{m};\frk{N},\bar{n}):\bar{m}\models p,\bar{n}\models q\}.$$
We will typically drop the $\lambda$ when it is clear from context.
\end{defn}

We're using $\delta$ instead of $d$ to emphasize that $\delta$ is not the natural counterpart of the ordinary $d$-metric. Instead it is the natural counterpart of
\[
\delta(p,q) = \begin{cases} 0 & p=q \\
				 \infty & p \neq q
				 \end{cases},
 \]
  i.e.\ a metric encoding equality of types. Later on there will be a metric, $d_\Delta$, derived from $\delta_\Delta$ that plays an analogous role to the $d$-metric on types. In some very special cases, such as Gromov-Hausdorff distance or Kadets distance, we will have $\delta_\Delta = d_\Delta$. 
  This in turn will entail some nice properties of $\Delta$-approximate isomorphism.

$\delta_\Delta^\lambda$ enjoys the following properties.

\begin{prop} Let $\Delta$ be a distortion system for $T$.
\begin{itemize}
    \item[(i)] $\delta^\lambda_\Delta(p,q)=\sup_{\varphi \in \Delta}|\varphi(p)-\varphi(q)|$, where $\varphi(r)$ means the unique value of $\varphi$ entailed by the type $r$.
    \item[(ii)] $\delta^\lambda_\Delta$ is a topometric on $S_\lambda(T)$, i.e.\ it is lower semi-continuous and refines the topology.
    \item[(iii)] (Monotonicity) For any $p,q\in S_{\lambda+\alpha}(T)$, if $p^\prime, q^\prime \in S_\lambda (T)$ are restrictions of $p$ and $q$ to the first $\lambda$ variables, then $\delta_\Delta^\lambda(p^\prime,q^\prime)\leq \delta_\Delta^{\lambda+\alpha}(p,q)$.
    \item[(iv)] For any $p,q \in S_\lambda(T)$ and any permutation $\sigma: \lambda \rightarrow \lambda$, $d_\Delta^n(p,q) = d_\Delta^n(\sigma p, \sigma q)$, where $\sigma r$ is the type $r(x_{\sigma(0)},x_{\sigma(1)},\dots)$.
    \item[(v)] (Extension) For any $p,q \in S_\lambda(T)$ and $p^\prime \in S_{\lambda + \alpha}(T)$ extending $p$, there exists a $q^\prime \in S_{\lambda + \alpha}(T)$ extending $q$ such that $d_\Delta^\lambda(p,q) = d_\Delta^{\lambda + \alpha}(p^\prime, q^\prime)$.
    \item[(vi)] For any infinite $\lambda$, $\delta^\lambda_\Delta(p,q) = \sup \delta^n_\Delta(p^\prime,q^\prime)$, where $p^\prime$ and $q^\prime$ range over restrictions of $p$ and $q$ to finite tuples of variables.
\end{itemize}
\end{prop}

\begin{proof}
\emph{(ii)}-\emph{(iv)} and \emph{(vi)} all follow immediately from \emph{(i)}.

It will be easier to prove \emph{(i)} once we have \emph{(v)}. To see that \emph{(v)} holds, let $(\frk{M}_i,\bar{m}_i,\frk{N}_i,\bar{n}_i,R_i)$ be structures such that $\bar{m}\models p$, $\bar{n}\models q$, $R_i$ is a correlation between $\frk{M}_i$ and $\frk{N}_i$ with $(\bar{m}_i,\bar{n}_i)\in R_i$ and $\mathrm{dis}_\Delta(R_i) \leq \delta^\lambda_\Delta(p,q)+2^{-i}$. By taking an ultraproduct of these structures we get an exact witness, i.e.\ a structure $(\frk{M},\bar{m},\frk{N},\bar{n},R)$ with $\bar{m}\models p$, $\bar{n}\models q$, and $R$, a correlation between $\frk{M}$ and $\frk{N}$ with $\mathrm{dis}_\Delta(R)=\delta^\lambda_\Delta(p,q)$.

Now by compactness $(\frk{M},\bar{m},\allowbreak\frk{N},\bar{n},R)$ has a $\aleph_1$-saturated elementary extension $(\frk{M}^\prime,\bar{m}^\prime,\frk{N}^\prime,\bar{n}^\prime,R^\prime)$ in which $\frk{M}^\prime$ realizes $p^\prime(\bar{x},\bar{m}^\prime)$ with some tuple $\bar{a}$. By $\aleph_1$-saturation, $R^\prime$ is a correlation, so we have that there is some tuple $\bar{b}$ with $(\bar{a},\bar{b})\in R^\prime$. So we can take $q^\prime = \mathrm{tp}(\bar{n}^\prime \bar{b})$ and get the required extension.

Now for \emph{(i)}. It is clear that $\delta^\lambda_\Delta(p,q) \geq \sup_{\varphi \in \Delta}|\varphi(p)-\varphi(q)|$. All we need, given $p,q$ with $\sup_{\varphi \in \Delta}|\varphi(p)-\varphi(q)| = \e$, is to build a structure $(\frk{M},\bar{m},\frk{N},\bar{n},R)$  witnessing that $\delta^\lambda_\Delta(p,q)\leq \e$, but this is almost immediate from the extension property, by a `back-and-forth Henkin construction.' So we have shown \emph{(i)} as well.
\end{proof}

As it happens, given a family of metrics $\{\delta^n\}_{n<\omega}$ satisfying some of these properties we can find a distortion system giving the same metrics. Again this is for a single-sorted theory but the extension to many-sorted theories is obvious.

Metrics satisfying these properties are very similar to  the `perturbation \mbox{(pre-)}systems' of \cite{OnPert}, but what we require here is more than a perturbation pre-system and less than a perturbation system.

\begin{prop} \label{prop:recon-met}
Suppose that $\{\delta^n\}_{n<\omega}$ is a family of topometrics on $S_n(T)$ such that:
\begin{itemize}
     \item (Monotonicity) For any $p,q\in S_{n+1}(T)$, if $p^\prime, q^\prime \in S_n (T)$ are restrictions of $p$ and $q$ to the first $n$ variables, then $\delta^n(p^\prime,q^\prime)\leq \delta^{n+1}(p,q)$.
    \item For any $p,q \in S_n(T)$ and any permutation $\sigma: n \rightarrow n$, $\delta^n(p,q) = \delta^n(\sigma p, \sigma q)$, where $\sigma r$ is the type $r(x_{\sigma(0)},x_{\sigma(1)},\dots,x_{\sigma(n-1)})$.
    \item (Extension) For any $p,q \in S_n(T)$ and $p^\prime \in S_{n+1}(T)$ extending $p$ there exists a $q^\prime \in S_{n+1}(T)$ extending $q$ such that $\delta^n(p,q) = \delta^{n+1}(p^\prime, q^\prime)$.
\end{itemize}

Then there is a distortion system $\Delta(\delta)$ (namely the collection of $\delta$-$1$-Lipschitz formulas) such that $\delta = \delta_{\Delta(\delta)}$.

Furthermore for any distortion system $\Delta$ we have that $\Delta = \Delta(\delta_\Delta)$.
\end{prop}

\begin{proof}
Let $\Delta(\delta)$ be the collection of formulas that are $1$-Lipschitz with regards to $\delta$ (in the relevant variables). Note that by the monotonicity property $\Delta(\delta)$ is closed under adding dummy variables (i.e.\ if $\varphi(\bar{x})$ is $\delta$-$1$-Lipschitz in $S_n(T)$, then it is $\delta$-$1$-Lipschitz in $S_{n+1}(T)$). $\Delta(\delta)$ is also clearly closed under renaming variables, $1$-Lipschitz connectives, logical equivalence modulo $T$, and uniformly convergent limits. So all we need to do is show that $\Delta(\delta)$ is closed under quantification and that it is logically complete. 

By a result of Ben Yaacov \cite{BYTopo2010}, if $(X,d)$ is a compact topometric space and $F,G\subseteq X$ are disjoint closed sets with $d_{\inf}(F,G) \coloneqq \inf \{d(x,y):x\in F, y\in G\} > \e$, then there is a $1$-Lipschitz continuous function $f:X\rightarrow[0,\e]$ such that $F \subseteq f^{-1}(0)$ and $G \subseteq f^{-1}(\e)$. This in particular implies that for any type $p\in S_n(T)$, $p$ is determined entirely by $\{\varphi:p\vdash \varphi \in \Delta(\delta)\}$, i.e.\ that $\Delta(\delta)$ is logically complete. Another corollary of his result, as pointed out by him, is that $d(x,y)=\sup\{|f(x)-f(y)|:f:X\rightarrow\mathbb{R}\,1\text{-Lipschitz continuous}\}$ in any compact topometric space $X$, so we clearly have that $\delta^n = \delta^n_{\Delta(\delta)}$ for each $n<\omega$.

To see that $\Delta(\delta)$ is closed under quantification, let $\varphi(\bar{x},y)\in \Delta(\delta)$. It is sufficient to show that $\inf_y \varphi(\bar{x},y) \in \Delta(\delta)$. For any $p,q\in S_n(T)$ (where $n=|\bar{x}|$), find a realization $\bar{a}\models p$ in the monster model and then find $b$ such that $\models \varphi(\bar{a},b) = \inf_y \varphi(\bar{a},y)$. Let $p^\prime = \mathrm{tp}(\bar{a}b)$. Find $q^\prime$ extending $q$ such that $\delta(p^\prime,q^\prime)=\delta(p,q)$. Now we have that $|\varphi(p^\prime)-\varphi(q^\prime)|\leq \delta(p^\prime,q^\prime) = \delta(p,q)$, implying that $q^\prime(\bar{x},y) \models \varphi(\bar{x},y) \leq \varphi(p^\prime) + \delta(p,q)$. This implies that $q(\bar{x}) \models \inf_y\varphi(\bar{x},y) \leq \varphi(p^\prime) + \delta(p,q)$, so by symmetry we have that $|\inf_y\varphi(p,y)-\inf_y\varphi(q,y)|\leq \delta(p,q)$. Since we can do this for any $p,q$ we have that $\inf_y\varphi(\bar{x},y)$ is $\delta$-$1$-Lipschitz and $\inf_y\varphi(\bar{x},y)\in \Delta(\delta)$. 

Therefore $\Delta(\delta)$ is a distortion system.

For the furthermore part, we clearly have that every $\Delta$-formula  is $\delta_\Delta$-$1$-Lipschitz. We just need to show that every $\delta_\Delta$-$1$-Lipschitz formula is a $\Delta$\nobreakdash-formula.

Let $\varphi(\bar{x})$ be a $\delta_\Delta$-$1$-Lipschitz formula. Pick $p\in S_n(T)$ (where $n=|\bar{x}|$). We have that $\varphi(q) \leq \varphi(p) + \delta(p,q)$ for all $q$, so in particular 
$$\varphi(q) \leq \varphi(p) + \sup_{\psi \in \Delta}|\psi(p)-\psi(q)|,$$
for all $q$. For any $\e>0$, by compactness there must be a finite set $\{\psi_1,\dots,\psi_k\} \subset \Delta$ such that $$\varphi(q) \leq \varphi(p) + |\psi_1(p)-\psi_1(q)|\uparrow\dots\uparrow|\psi_k(p)-\psi_k(q)| + \e,$$
for all $q$. $|\psi_1(p)-\psi_1(\bar{x})|\uparrow\dots\uparrow|\psi_k(p)-\psi_k(\bar{x})| + \e$  is a $\Delta$-formula, so we have shown that 
$$\varphi(q) = \inf\{\psi(q) : \psi\in \Delta, \psi \geq \varphi\},$$
for all $q$. Now for each $i<\omega$, by compactness there must be a finite set $\{\psi_1^i,\dots,\allowbreak\psi_{k(i)}^i\} \subset \Delta$ such that 
$$\varphi(q) \leq \psi_1(q) \downarrow \dots \downarrow \psi_k(q) < \varphi(q) + 2^{-i}  $$
for all $q$.

So if we let $\chi_0 = \psi^0_1 \downarrow \dots \downarrow \psi^0_{k(0)}$ and $\chi_{j+1} = \chi_j \downarrow \psi^{j+1}_1 \downarrow \dots \downarrow \psi^{j+1}_{k(j+1)}$, we get that $\{\chi_j\}_{j<\omega}$ is a sequence of $\Delta$-formulas that uniformly converges to $\varphi$, so $\varphi\in\Delta$, as required.
\end{proof}


\subsection{Theories of Approximate Isomorphism}
\label{sec:th-approx-iso}

Implicit in a lot of the arguments so far has been the fact that if $\Delta$ is a distortion system for $T$, then for any $\e$ there is a first-order theory whose models are precisely structures $(\frk{M},\frk{N},R)$, with $R$ a closed almost correlation between $\frk{M}$ and $\frk{N}$ such that $\mathrm{dis}_\Delta(R) \leq \e$. This is how notions of approximate isomorphism are typically presented, at least implicitly. There is some kind of ambient structure relating $\frk{M}$ and $\frk{N}$ witnessing a certain degree of closeness, such as a mutual embedding into a larger structure or a certain special kind of function between them. We will give a precise characterization of these theories in our context and show that $\Delta$ can be reconstructed from them.

\begin{defn}
If $\Delta$ is a distortion system for $T$, the for any $\e\in[0,\infty]$ let $\mathrm{Th}(\Delta,\e)$ be the common theory of all structures of the form $(\frk{M},\frk{N},R)$ with $\frk{M},\frk{N}\models T$ and closed $R\in\mathrm{cor}(\frk{M},\frk{N})$ with $\mathrm{dis}_\Delta(R)\leq \e$, where $R$ is taken as a family of definable subsets of $s(\frk{M})\times s(\frk{N})$ for sorts $s\in\mathcal{S}$.
\end{defn}

\begin{prop}
Let $\Delta$ be a distortion system for $T$. For any $\e\in [0,\infty]$, a triple $(\frk{M},\frk{N},R)\models \mathrm{Th}(\Delta,\e)$ if and only if $R$ is a closed almost correlation between $\frk{M}$ and $\frk{N}$ and $\mathrm{dis}_\Delta(R)\leq \e$.
\end{prop}

\begin{proof}
$(\Rightarrow):$ Assume that $(\frk{M},\frk{N},R)\models \mathrm{Th}(\Delta,\e)$. Clearly we have that $\frk{M},\frk{N}\models T$ and that for all $(\bar{n},\bar{m})\in R$, $\delta_\Delta(\mathrm{tp}(\bar{n}),\mathrm{tp}(\bar{m}))\leq \e$, so the only thing to really check is that $R$ is an almost correlation. This follows because it is equivalent to the following first-order axiom schema:
    $$\sup_{x\in s(\frk{M})}\inf_{y\in s(\frk{N})} R_s(x,y)\uparrow \sup_{y\in s(\frk{N})} \inf_{x\in s(\frk{M})} R_s(x,y),$$
for each sort $s\in\mathcal{S}$, where $R_s$ is the distance  predicate for the set $R\upharpoonright s$.

$(\Leftarrow):$ Take an $\aleph_1$-saturated elementary extension of $(\frk{M}^\prime,\frk{N}^\prime,R^\prime) \succeq (\frk{M},\frk{N},R)$. By $\aleph_1$-saturation, $R^\prime$ is a closed correlation between $\frk{M}^\prime$ and $\frk{N}^\prime$, and we still have that $\mathrm{dis}_\Delta(R^\prime)\leq \e$, so by definition $(\frk{M}^\prime,\frk{N}^\prime,R^\prime)\models \mathrm{Th}(\Delta,\e)$, thus by elementarity, $(\frk{M},\frk{N},R)\models \mathrm{Th}(\Delta,\e)$. 
\end{proof}

Unsurprisingly, $\Delta$ is recoverable from the theories $\mathrm{Th}(\Delta,\e)$. 
\begin{prop} \label{prop:recover}
Fix a theory $T$ and suppose that $\{A_\e\}_{\e\in [0,\infty]}$ is a family of first-order theories that satisfy the following conditions:
\begin{itemize}
    \item For every $\e$, every model of $A_\e$ is of the form $(\frk{M},\frk{N},R)$ with $\frk{M}$ and $\frk{N}$ models of $T$ where $R$ is a family of distance predicates $R_s$ on $s(\frk{M})\times s(\frk{N})$.
    \item A triple $(\frk{M},\frk{N},R)$ is a model of $A_\infty$ if and only if $R$ is a closed almost correlation between $\frk{M}$ and $\frk{N}$, namely if for each sort $s$, it satisfies 
   $$ \sup_{x\in s(\frk{M})}\inf_{y\in s(\frk{N})} R_s(x,y)\uparrow \sup_{y\in s(\frk{N})} \inf_{x\in s(\frk{M})} R_s(x,y).$$
    \item A triple $(\frk{M},\frk{N},R)$ is a model of $A_0$ if and only if $R$ is the graph of an isomorphism between $\frk{M}$ and $\frk{N}$.
    \item For each $\varepsilon < \delta$, $A_\varepsilon$ logically entails $A_\delta$ and $\bigcup_{\delta > \e}A_\delta$ is logically equivalent to $A_\e$.
    \item (Symmetry) If $(\frk{M},\frk{N},R)\models A_\e$, then $(\frk{N},\frk{M},R^{-1})\models A_\e$, where $R^{-1} \coloneqq \{(y,x):(x,y) \in R\}$.
    \item (Composition) For every $\e,\delta > 0$ if $(\frk{M},\frk{N},R)\models A_\e$ and $(\frk{N},\frk{O},S)\models A_\delta$ and $(\frk{M},\frk{N},R)$ and $(\frk{N},\frk{O},S)$ are $\aleph_1$-saturated, then $(\frk{M},\frk{O},\ol{S \circ R})\models A_{\e + \delta}$, where $\ol{S \circ R}$ is understood to mean the family of distance predicates of the metric closure of the relation $S \circ R$.
    \item (Sub-structure) If $(\frk{M},\frk{N},R)\models A_\e$ and $\frk{M}^\prime \preceq \frk{M}$ and $\frk{N}^\prime \preceq \frk{N}$ are elementary sub-structures such that $(\frk{M},\frk{N},R)$, $\frk{M}^\prime$, and $\frk{N}^\prime$ are all $\aleph_1$-saturated, and $R^\prime = R\upharpoonright \frk{M}^\prime \times_\mathcal{S} \frk{N}^\prime$ is a correlation, then $(\frk{M}^\prime,\frk{N}^\prime, R^\prime)\models A_\e$.
\end{itemize}
Then there is a distortion system $\Delta$ such that $A_\e \equiv \mathrm{Th}(\Delta,\e)$ for every $\e\in [0,\infty]$.
\end{prop}
\begin{proof}
First we will show that $\{A_\e\}_{\e\in[0,\infty]}$ induces a family of topometrics $\{\delta^n_A\}_{n<\omega}$ satisfying the conditions of Proposition \ref{prop:recon-met}. Then we will show that $A_\e =\allowbreak \mathrm{Th}(\Delta(\delta),\e)$ for every $\e\in [0,\infty]$.

Let $\delta_A(p,q) = \inf\{\e:(\frk{M},\frk{N},R)\models A_\e, R\in \mathrm{cor}(\frk{M},\bar{m};\frk{N},\bar{n}), \bar{m} \models p, \bar{n} \models q \}$.

It is clear that $\delta_A(p,q) \geq 0$ and that $\delta_A(p,p) = 0$. By symmetry we have that $\delta_A(p,q)=\delta_A(q,p)$.

Pick $p,q,r\in S_n(T)$. Pick $(\frk{M},\frk{N},R)$ witnessing $\delta_A(p,q) \leq \e$ and $(\frk{N}^\prime, \frk{O},S)$ witnessing that $\delta_A(q,r) \leq \delta$. By passing to elementary extensions we can find triples $(\frk{M}^\prime,\frk{N}^{\prime \prime},R^\prime)$ and $(\frk{N}^{\prime \prime},\frk{O}^\prime,S^\prime)$ with tuples $\bar{m}$ and $\bar{n}$ such that $\frk{M}^\prime \ni \bar{m} \models p$ and $\frk{N}^{\prime\prime}\ni \bar{n}\models q$, and $\frk{O}^\prime \ni \bar{o}\models r$ such that all structures involved are $\aleph_1$-saturated. By composition we have that $(\frk{M}^\prime,\frk{O}^\prime,\ol{S\circ R})\models A_{\e+\delta}$, witnessing that $\delta(p,r)\leq \e+\delta$. Since we can do this for any $\e$ and $\delta$, we have that $\delta(p,r)\leq \delta(p,q)+\delta(q,r)$.

Finally by taking ultraproducts of witnesses, it is clear that if $\delta(p,q)=0$ then $p=q$.

So we have that $\delta^n$ are metrics. They are clearly lower semi-continuous, again by taking ultraproducts of relevant witnesses, so they are a family of topometrics. Now we just need to verify the other conditions of Proposition \ref{prop:recon-met}. Monotonicity and permutation invariance are both clear. For extension, suppose that $(\frk{M},\frk{N},R)$ is an exact witness for the value of $\delta^n(p,q)$, i.e.\ there are $\bar{m}\in \frk{M}$ and $\bar{n}\in\frk{N}$ such that $\bar{m}\models p$ and $\bar{n}\models q$ and $(\bar{m},\bar{n})\in R$, a correlation. Then by passing to a saturated enough elementary extension we can find $a$ such that $\models p^\prime(\bar{m},a)$. By picking some $b$ correlated to $a$ and picking $q^\prime = \mathrm{tp}(\bar{n}b)$, we get the required extension.

So we have that Proposition \ref{prop:recon-met} applies and $\Delta(\delta_A)$ is a distortion system with $\delta_A = \delta_{\Delta(\delta_A)}$.

So now clearly by construction we have that for any $\e \in [0,\infty]$, $A_\e \vdash \mathrm{Th}(\Delta(\delta_A),\e)$. So all we need to do is show that $\mathrm{Th}(\Delta(\delta_A),\e)\vdash A_\e$. Let $(\frk{M},\frk{N},R)\models \mathrm{Th}(\Delta(\delta_A),\e)$. Assume that is $(\frk{M},\frk{N},R)$ is $\aleph_1$-saturated, by passing to an elementary extension if necessary. By construction for every pair of finite tuples, $(\bar{m},\bar{n}) \in R$ there exists $(\frk{A}_{\bar{m},\bar{n}},\frk{B}_{\bar{m},\bar{n}},S_{\bar{m},\bar{n}})\models A_\e$ such that $(\bar{m},\bar{n})\in S_{(\bar{m},\bar{n})}$. Let $\mathcal{F}$ be the filter on $R^{<\omega}$ ordered by extensions of tuples and let $\mathcal{U}$ be an ultrafilter extending $\mathcal{F}$. Take the ultraproduct $(\frk{A}^\prime,\frk{B}^\prime,S^\prime)=\prod_{(\bar{n},\bar{m})\in R^{<\omega}} (\frk{A}_{\bar{m},\bar{n}},\frk{B}_{\bar{m},\bar{n}},S_{\bar{m},\bar{n}}) /\mathcal{U}$ and assume that this is $\aleph_1$-saturated (taking an elementary extension if necessary). By construction we have that $\frk{M}\preceq \frk{A}^\prime$, $\frk{N}\preceq \frk{B}^\prime$, and $R = S^\prime \upharpoonright \frk{M}\times_\mathcal{S} \frk{N}$ is a correlation, so by the sub-structure property we have that $(\frk{M},\frk{N},R)\models A_\e$.
Since we can do this for any theory completing $\mathrm{Th}(\Delta,\e)$, we have that $\mathrm{Th}(\Delta,\e)\vdash A_\e$, so $\mathrm{Th}(\Delta,\e) = A_\e$ as required.
\end{proof}

Now we can finally tie up a loose end. We will need the following fact.

\begin{fact}\label{fact:Approx_Iso_in_Cts_Logic:1}
A definable set is the graph of a definable function if and only if it is the graph of a function in every model. It is sufficient to check $\aleph_1$-saturated models.
\end{fact}

\begin{customprop}{\ref{prop:nice-regular}}
\hypertarget{repeated}{Let $\Delta$ be a distortion system.}
\begin{itemize}
    \item [(ii)] $\Delta$ is functional if and only if there is an $\e>0$ such that for any $\delta > 0$ there is a formula $\varphi(x,y)\in \Delta$ such that for any $\frk{M}\models T$ and  $a,b\in \frk{M}$, $\varphi^\frk{M}(a,a)=0$ and if $\varphi^\frk{M}(a,b) < \e$, then $d^\frk{M}(a,b)<\delta$.
\end{itemize}
\end{customprop}

\begin{proof}

$(\Rightarrow):$ For any $\varphi(x.y) \in \Delta$, let $\chi_\varphi(x,y) =  \frac{1}{2}|\varphi(x,y)-\varphi(x,x)|$ and note that $\chi_\varphi$ is always a $\Delta$-formula.

Assume that for every $\e > 0$ there exists a $\delta > 0$ such that for any $\varphi \in \Delta$ there exists $a,b\in\frk{M}\models T$, either $\varphi^\frk{M}(a,a)\neq 0$ or ($\varphi^\frk{M}(a,b) < \e$ and $d^\frk{M}(a,b) \geq \delta$).  

In particular this implies that for every $\e > 0$ there exists a $\delta >0$ such that for any $\varphi_1,\dots,\varphi_k \in \Delta$ there exists $a,b\in \frk{M}\models T$, $\chi_{\varphi_1}^\frk{M}(a,b) \uparrow \dots \uparrow \chi_{\varphi_k}^\frk{M}(a,b) < \e$ and $d^\frk{M}(a,b)\geq \delta$. 

By compactness this implies that for any $\e>0$ there is a $\delta >0$ and $c,e\in\frk{N}\models T$ such that for every $\varphi \in \Delta$, $\chi_\varphi^\frk{N}(c,e)\leq \e$ and $d^\frk{M}(c,e)\geq \delta$. In particular this implies that $\delta_\Delta(\mathrm{tp}(ce),\mathrm{tp}(cc))\leq 2\e$. So we can build a structure witnessing this and we have that $\Delta$ cannot be functional at $\e > 0$. Since we can do this at any $\e >0$, $\Delta$ is not functional.

$(\Leftarrow):$ Assume that there is an $\e>0$ such that for any $\delta > 0$ there is a formula $\varphi_\delta(x,y)\in \Delta$ such that for any $a,b\in \frk{M}\models T$, if $\varphi^\frk{M}_\delta(a,b) < \e$, then $d^\frk{M}(a,b)<\delta$. Pick $0<\gamma<\e$ and let $(\frk{M},\frk{N},R)$ be an $\aleph_1$-saturated model of $\mathrm{Th}(\Delta,\gamma)$. For each $a,b\in\frk{M}$ and $c\in\frk{N}$ with $(a,c),(b,c)\in R$, we have that $|\varphi_\delta^\frk{M}(a,b)-\varphi_\delta^\frk{N}(c,c)|\leq \gamma$, so in particular $\varphi_\delta^\frk{M}(a,b) \leq \gamma < \e$. So we have that $d^\frk{M}(a,b)<\delta$. Since we can do this for any $\delta > 0$, we have that $d^\frk{M}(a,b)=0$ and $a=b$.

Therefore $R$ is the graph of a bijection in every $\aleph_1$-saturated model of $\mathrm{Th}(\Delta, \gamma)$. This implies that it is actually the graph of a definable bijection, so this fact must be true in every model of $\mathrm{Th}(\Delta,\gamma)$. By compactness this implies that there is a modulus $\alpha_\gamma$ such that in every model of $\mathrm{Th}(\Delta,\gamma)$, $R$ and $R^{-1}$ are $\alpha_\gamma$-uniformly continuous. So we have that every closed $R\in \mathrm{acor}(\frk{M},\frk{N})$ with $\mathrm{dis}_\Delta(R) < \e$ is the graph of a uniformly continuous bijection with uniformly continuous inverse, therefore $\Delta$ is functional.
\end{proof}

A corollary of this is that when checking functionality of $\Delta$ it is enough to check closed correlations, rather than closed almost correlations.

\section{Special Cases}

Here we will examine a few specific cases of notions of approximate isomorphism arising from distortion systems.

\subsection{Elementary and Finitary Gromov-Hausdorff-Kadets Distances}

\begin{defn}
Let $\Delta_0$ and $\Delta_1$ be distortion systems.

We say that $\Delta_1$ \emph{uniformly dominates} $\Delta_0$ if for every $\e>0$ there is a $\delta > 0$ such that if $\delta_{\Delta_1}(p,q) < \delta$ then $\delta_{\Delta_0}(p,q) < \e$. We may also say that $\Delta_0$ is \emph{coarser} than $\Delta_1$ or that $\Delta_1$ is \emph{finer} than $\Delta_0$.

If $\Delta_0$ and $\Delta_1$ uniformly dominate each other we say that they are \emph{uniformly equivalent}.
\end{defn}

Note that $\Delta_1$ uniformly dominates $\Delta_0$ if and only the collection of $\Delta_0$\nobreakdash-\hskip0pt formulas are u.u.c.\ with regards to $\delta_{\Delta_1}$. 

\begin{prop} \label{prop:fine-coarse}
Fix a signature $\Lcal$.
\begin{itemize}
    \item[(i)] There is a collection of formulas, $\mathrm{eGHK}_0$, such that for any $\Lcal$-theory $T$, $\mathrm{eGHK}_0$ generates the finest u.u.c.\ distortion system for $T$, up to uniform equivalence. Furthermore $\delta_{\mathrm{eGHK}_0} = d$, the $d$-metric on types.
    \item[(ii)] If $\Lcal$ is countable then there is a collection of formulas, $\mathrm{fGHK}_0$, such that for any $\Lcal$-theory $T$, $\mathrm{fGHK}_0$ generates the coarsest distortion system for $T$, up to uniform equivalence.
\end{itemize} 
\end{prop}
\begin{proof}
\emph{(i):} For any $\Lcal$-formula $\varphi$, let $\chi_\varphi(\bar{x}) = \inf_{\bar{y}}\varphi(\bar{y})+d(\bar{x},\bar{y})$.  $\chi_\varphi$ has the property that it is $1$-Lipschitz in any $\mathcal{L}$-structure and furthermore that if $\varphi$ is $1$-Lipschitz in every model of $T$, then $T\models \chi_\varphi = \varphi$. Let $\mathrm{eGHK}_0 =\{\chi_\varphi : \varphi \in \Lcal\}$. Note that $\varphi \in \mathrm{eGHK}_0$ for any sentence $\varphi$.

By a previously mentioned result of Ben Yaacov \cite{BYTopo2010}, for any types $p,q$ in the same complete theory, $d(p,q) = \delta_{\mathrm{eGHK}_0}(p,q)$ (and for types in different complete theories $\delta_{\mathrm{eGHK}_0}(p,q) = \infty$). This implies that $\mathrm{eGHK}_0$ uniformly dominates any u.u.c.\ distortion system. 

\emph{(ii):} Let $\{P_i\}_{i<\omega}$ be an enumeration of all atomic $\mathcal{L}$-predicates. For each $i$, if $P_i$'s syntactic range is $[a,b]$, let $r_i = 1 + |a|\uparrow |b|$. 

Let $\mathrm{fGHK}_0 = \left\{\frac{1}{2^i r_i}  P_i\right\}_{i<\omega}$. This is clearly atomically complete. By Lemma \ref{lem:coarsest}, in any theory $T$, any distortion system for $T$ uniformly dominates $\ol{\mathrm{fGHK}_0}$.
\end{proof}

\begin{defn}
Fix a theory $T$ and an enumeration of atomic $\Lcal$-formulas.
\begin{itemize}
    \item $\mathrm{eGHK}= \ol{\mathrm{eGHK}_0}$, as defined in the proof of Proposition \ref{prop:fine-coarse}. $\rho_{\mathrm{eGHK}}(\frk{M},\frk{N})$ is the \emph{elementary Gromov-Hausdorff-Kadets distance between $\frk{M}$ and $\frk{N}$}.
    \item If $T$ is countable, let $\mathrm{fGHK}=\ol{\mathrm{fGHK}_0}$, as defined in the proof of Proposition \ref{prop:fine-coarse}. $\rho_{\mathrm{fGHK}}(\frk{M},\frk{N})$ is `the' \emph{finitary Gromov-Hausdorff-Kadets distance between $\frk{M}$ and $\frk{N}$}. \qedhere
\end{itemize}
\end{defn}

Clearly $\rho_{\mathrm{fGHK}}$ depends on the choice of enumeration, but $\approxx_{\mathrm{fGHK}}$ does not.

\begin{prop} \label{prop:fGHK}
  $\frk{M}\approxx_{\mathrm{fGHK}} \frk{N}$ if and only if for every $\e > 0$, finite collection $\mathcal{S}_0 \subseteq \mathcal{S}$ of sorts, and finite collection $\Sigma$ of atomic $\mathcal{L}$-formulas whose variables are from sorts in $\mathcal{S}_0$, there exists a correlation $R \subseteq \frk{M}\times_{\mathcal{S}_0}\frk{N}$ such that $\mathrm{dis}_{\Sigma}(R) < \e$.
\end{prop}

You may have noticed that the condition in Proposition \ref{prop:fGHK} makes sense in uncountable languages. Indeed there is a canonical uniform structure analog of $\rho_{\mathrm{fGHK}}$, given by a family of pseudo-metrics. It is possible to develop the whole theory of distortion systems with this more general context in mind, similarly to \cite{Iovino95stablebanach}. Rather than a single collection of formulas we would need a directed family of collections of formulas. In the absence of motivating examples we opted to develop this simpler framework.

Clearly, for the empty signature, $\rho_{\mathrm{fGHK}}$ is uniformly equivalent to $\rho_{\mathrm{GH}}$. This is true of $\rho_K$ in the theory of (unit balls of) Banach spaces as well, justifying the name.

\begin{prop}
Let $T$ be the theory of unit balls of Banach spaces. $\rho_{\mathrm{fGHK}}$ is uniformly equivalent to $d_K$, the Kadets distance.
\end{prop}

\begin{proof}
Fix two unit ball Banach space structures $\frk{M}$ and $\frk{N}$.

\emph{Claim:} For every $\e>0$, there is a $\delta>0$ such that if $\rho_{\mathrm{fGHK}}(\frk{M},\frk{N}) < \delta$, then there exists a correlation $R\in \mathrm{cor}(\frk{M},\frk{N})$ such that:
\begin{itemize}
    \item $(\mathbf{0}^\frk{M},\mathbf{0}^\frk{N})\in R$
    \item If $(a,c),(b,d)\in R$, then $(\frac{1}{2}(a+b),\frac{1}{2}(c+d)) \in R$.
    \item If $(a,b)\in R$ then $(-a,-b)\in R$.
    \item (For complex Banach spaces) If $(a,b)\in R$ then $(e^ia,e^ib)\in R$.
    \item If $(a,c),(b,d)\in R$, then $|\lVert a - b \rVert_\frk{M} - \lVert c-d \rVert_\frk{N}| \leq \e$.
\end{itemize}

\emph{Proof of claim:} Fix $\e > 0$. There is a $\delta > 0$ such that for any $S\in \mathrm{cor}(\frk{M},\frk{N})$ with $\mathrm{dis}_{fGHK}(S) \leq \delta$, then:

\begin{itemize}
    \item For every $(a,b)\in S$, $|\lVert a \rVert_\frk{M} - \lVert b \rVert_\frk{N}| \leq \frac{1}{5}\e$.
    \item For every $(a,c),(b,d)\in S$, $|\lVert a - b \rVert_\frk{M} - \lVert c-d \rVert_\frk{N}| \leq \frac{1}{5}\e$, $|\lVert a + b \rVert_\frk{M} - \lVert c+d \rVert_\frk{N}| \leq \frac{1}{5}\e$, and (if the Banach spaces are complex) $|\lVert e^i a - b \rVert_\frk{M} - \lVert e^i c-d \rVert_\frk{N}| \leq \frac{1}{5}\e$.
    \item For every $(a,d),(b,e),(c,f)\in S$, $|\lVert \frac{1}{2}(a+b) - c \rVert_\frk{M} - \lVert \frac{1}{2}(d+e)-f \rVert_\frk{N}| \leq \frac{1}{5}\e$.
\end{itemize}

Such a $\delta$ exists because this is a finite list of atomic formulas. Consider the correlation $$R=\left\{(a,b):(\exists (c,d) \in S)\lVert a-c\rVert_\frk{M}\leq \frac{1}{5}\e\wedge\lVert b-d\rVert_\frk{N} \leq \frac{1}{5}\e\right\}.$$ 
Now we have what we want:

\begin{itemize}
    \item $(\mathbf{0}^\frk{M},\mathbf{0}^\frk{N})$ is in $R$ because if $(a,\mathbf{0}^\frk{N})\in S$, then the distance between $a$ and $\mathbf{0}^\frk{M}$ is $\leq \frac{1}{5}\e$.
    \item If $(a,c),(b,d),(\frac{1}{2}(a+b),e)\in S$, then the distance between $e$ and $\frac{1}{2}(c+d)$ is $\leq \frac{1}{5}\e$.
    \item If $(a,b),(-a,c)\in S$, then the distance between $c$ and $-b$ is $\leq \frac{1}{5}$.
    \item (For a complex Banach space) If $(a,b),(e^i a, c)\in S$, then the distance between $c$ and $e^i b$ is $\leq \frac{1}{5}\e$.
    \item If $(a,c),(b,d)\in R$, then there are $(a^\prime,c^\prime),(b^\prime,d^\prime) \in S$ each distance $\frac{1}{5}$ to the corresponding element. By several applications of the triangle inequality this implies that $|\lVert a - b \rVert_\frk{M} - \lVert c-d \rVert_\frk{N}| \leq \frac{4}{5}\e < \e$.\hfill $\qed_{\text{claim}}$
\end{itemize}

By iterating the second bullet point in the claim we get the following: For any $n<\omega$, and any $(a_1,b_1),\dots,(a_{2^n},b_{2^n}) \in R$, $\left(2^{-n}\sum_i a_i, 2^{-n}\sum_i b_i\right) \in R$. By using duplicates  we get that if $\lambda_1,\dots,\lambda_n$ are a sequence of positive dyadic rationals with $\sum_i \lambda_i = 1$, then for any $(a_1,b_1),\dots,(a_{m},b_{m}) \in R$, $\left(\sum_i \lambda_i a_i, \sum_i \lambda_i b_i\right) \in R$. 

Using the third and fourth bullet points we get that if $\lambda_1,\dots,\lambda_n$ are a sequence of numbers of the form $ad$ with $a = \pm e^{ik}$ with $k<\omega$ and $d$ a positive dyadic rational $\leq 1$, if $\sum_i |\lambda_i| = 1$, then for any $(a_1,b_1),\dots,(a_{m},b_{m}) \in R$, $\left(\sum_i \lambda_i a_i, \sum_i \lambda_i b_i \right) \in R$, so in particular $$\left|\left\lVert \sum_i \lambda_i a_i \right\rVert_\frk{M} - \left\lVert \sum_i \lambda_i b_i \right\rVert_\frk{N}\right| \leq \e. $$
Since $\lambda_i$ of this form are dense in the set of all coefficients $\gamma_i$ satisfying $\sum_i |\gamma_i|= 1$ (for complex Banach spaces this relies on the fact that $e^i$ is an irrational rotation) and by Fact 3.4 in \cite{10.2307/23809583}, this implies that $d_K(\frk{M},\frk{N})\leq \e$.

The other direction follows from the minimality of $\rho_{\mathrm{fGHK}}$ under uniform domination.
\end{proof}

\subsection{Banach-Mazur Distance}

Difficulty arises with the Banach-Mazur distance in that the witnessing correlations are bijections between the entire Banach spaces in question. To deal with this we will use Ben Yaacov's emboundment concept \cite{Yaacov2008-ITACFO} to encode the entire Banach space as a bounded structure.

We could in principle do this more cleanly using the full logic for unbounded structures in \cite{Yaacov2008-ITACFO}, but then we would have to re-develop the machinery of distortion systems in that broader context. We should note that Ben Yaacov does develop a theory of perturbations for unbounded metric structures in \cite{Yaacov2008-ITACFO}.

\begin{defn}
An \emph{embounded Banach space structure} is a metric structure $\{\frk{M},\allowbreak d,\allowbreak \mathbf{0},\allowbreak \infty,\allowbreak P,\allowbreak S_r\}_{r\in K}$, where $\frk{M}$ is a Banach space over the field $K\in\{\mathbb{R},\mathbb{C}\}$ together with an additional point $\infty$.

Let $\theta(x) = \frac{x}{1+x}$. The metric is $d^\frk{M}(x,y)=\frac{\theta(\lVert x - y \rVert)}{1+\lVert x \rVert \downarrow \lVert y \rVert}$ for $x,y\neq \infty$ and $d(x,\infty)=\frac{1}{1+\lVert x \rVert}$. $P(x,y,z)=\frac{\theta(\lVert x + y - z \rVert)}{1+\lVert x \rVert \uparrow \lVert y \rVert \uparrow \lVert z \rVert}$ if $x,y,z\neq \infty$, and $P(x,y,z)=0$ if any of $x,y,z$ are $\infty$. $S_r(x,y) = \frac{\theta(\lVert rx -  y\rVert)}{1+\lVert x \rVert \uparrow \lVert y \rVert}$ if $x,y\neq \infty$, and $S_r(x,y)=0$ if either $x$ or $y$ are $\infty$.
\end{defn}

Note that even though the language as stated is uncountable it is actually interdefinable with a finite sub-language\footnote{For $K=\mathbb{R}$, $S_{\frac{1}{2}}$ is sufficient, and for $K=\mathbb{C}$ we also need $S_i$.} in unit ball Banach space structures. 

In order to describe the formulas that capture the Banach-Mazur distance we will freely use the following facts:

\begin{fact} There is a theory whose models are precisely emboundments of Banach spaces. Let $T$ be that theory.
\begin{itemize}
    \item For any $r > 0$ there is a formula that is the distance predicate of the ball of (norm) radius $r$, $B_r$, in any model of $T$.
    \item For any $r > 0$ there is a formula that defines $\lVert x \rVert$ in $B_r$ in any model of $T$.
    \item For any $r>0$ there is a formula that defines the function $+:B_r^2\rightarrow B_{2r}$ in any model of $T$.
    \item For any $s\in K$ and any $r>0$ there is a formula that defines the function $(x\mapsto sx):B_{r}\rightarrow B_{|s|r}$ in any model of $T$.
\end{itemize}
\end{fact}

Also note that inclusion maps between definable sets are always uniformly definable.

This lemma follows immediately from the previous set of facts, although a careful proof would be slightly involved.

\begin{lem} \label{lem:BM-fmlas}
Let $T$ be the theory of emboundments of Banach space structures. There are formulas that define the following quantities in any model of $T$ for any real $r > 0$, and $t\in K$,
\begin{itemize}
    \item $\varphi_{r}(x,y,z) = [r - r^{-2}\log(\lVert x \rVert \uparrow \lVert y \rVert \uparrow \lVert z \rVert)]_0^1 \cdot [(2 - r^{-1})\log \lVert x+y-z \rVert]_{-r}^r$,
    \item $\psi_{r,s}(x,y)=[r - r^{-2}\log(\lVert x \rVert \uparrow \lVert y \rVert)]_0^1 \cdot [(2 - r^{-1})\log \lVert sx-y \rVert]_{-r}^r$,
\end{itemize}
where these quantities are understood to be $0$ if any of their inputs are $\infty$.
\end{lem}

To clarify what we're doing, intuitively we're after  expressions of the form $2\log\lVert \dots\rVert$, with $\dots$ replaced with various linear combinations, to capture the Banach-Mazur distance. These are unbounded so we need to use bounded approximations. Unlike with the Lipschitz metric, we can't just use $[2\log \lVert \dots \rVert]_{-r}^r$ as they aren't by themselves continuous on the emboundment (specifically the problem is at $\infty$). Given this, we need the more complicated expressions of the form $[\dots]_0^1$ as cutoff functions which are $1$ whenever the maximum norm of the inputs is less than $e^{r^3 - r^2}$ and which are $0$ whenever it is greater than $e^{r^3}$. The specific form of these cutoffs and the coefficient $2-r^{-1}$ are chosen so that the $(\Leftarrow)$ direction of the next result, Proposition \ref{prop:embound-cor}, will work. 

\begin{defn}
Let $\mathrm{BM}_0$ be the formulas in Lemma $\ref{lem:BM-fmlas}$ allowing substitution of the constant $\mathbf{0}$. Let $\mathrm{BM}=\ol{\mathrm{BM}_0}$.
\end{defn}

To see that $\mathrm{BM}_0$ is atomically complete, note that by choosing large enough values for $r$ (so that the cutoff function is $1$) and appropriate values of $s$ and $t$, the formulas in $\mathrm{BM}_0$ clearly fix the values of $d(x,y)$, $P(x,y,z)$, $S_r(x,y)$ for any $x,y,z\in\{a,b,c,\mathbf{0}\}$ with $a,b,c$ any triple of elements of a structure. The only unclear thing is determining the value of $d(a,\infty)$, but this $1-d(a,\mathbf{0})$, so $\mathrm{BM}_0$ is atomically complete. Therefore $\mathrm{BM}$ is a distortion system.

\begin{prop} \label{prop:embound-cor}
Let $\frk{M}$ and $\frk{N}$ be emboundments of the Banach spaces $X$ and $Y$, respectively. For $R \in \mathrm{cor}(\frk{M},\frk{N})$ a closed correlation, $\mathrm{dis}_{\mathrm{BM}}(R) \leq \e < \infty$ if and only if $R$ is the graph of a linear bijection between $X$ and $Y$ (together with the tuple $(\infty^\frk{M},\infty^\frk{N})$) such that $\lVert R \rVert \leq \sqrt{e^\e}$ and $\lVert R^{-1} \rVert \leq \sqrt{e^\e}$.
\end{prop}
\begin{proof}
$(\Rightarrow):$ Assume that $R$ is a closed correlation between $\frk{M}$ and $\frk{N}$ with $\mathrm{dis}_{\mathrm{BM}}(R) \allowbreak \leq \e < \infty$. 

Pick $m \in \frk{M}$ and assume that $(m,\infty^\frk{N})\in R$. Consider the formula $\psi_{r,s}(x,\mathbf{0})$. Since $\psi_{r,s}^\frk{N}(\infty^\frk{N},\mathbf{0}^\frk{N})=0$ for any $r,s$, we have that $|\psi_{r,s}^\frk{M}(m,\mathbf{0}^\frk{M})|\leq \e$ for any $r,s$. Assume that $a \neq \infty^\frk{M}$. When $r$ is large enough (relative to the choice of $s$) we have that $\psi_{r,s}^\frk{M}(m,\mathbf{0}^\frk{M}) = [(2-r^{-1})\log\lVert sm \rVert]_{-r}^{r}$, but this quantity is unbounded in $r$ and $s$ (even if $a=\mathbf{0}^\frk{M}$), so we must have that $a=\infty^\frk{M}$.

By symmetry $(\infty^\frk{M},\infty^\frk{N})$ is the only instance of a pair containing either copy of $\infty$.

Pick $a,b\in X$ and consider $d,e,f \in Y$ such that $(a,d),(b,e),(a+b,f)\in R$. For any sufficiently large $r$, we have that $\varphi_{r}^\frk{M}(a,b,a+b) = -r$. Assume that $f \neq d+e$. Then for any sufficiently large $r$, we have that $$\varphi_{r}^\frk{N}(d,e,f) = (2-r^{-1})\log \lVert b + e - f \rVert > \log \lVert b + e - f \rVert > -\infty.$$ Since we can choose $r$ arbitrarily large, this contradicts that $\mathrm{dis}_{\mathrm{BM}}(R)\leq \e$. Therefore $f=d+e$.

The same argument shows that if $(a,b),(ua,c)\in R$, then $b=ua$. In particular this implies that $(\mathbf{0}^\frk{M},\mathbf{0}^\frk{N})$ is the only correlation involving a copy of $\mathbf{0}$.

Therefore, by symmetry, $R\upharpoonright X\times Y$ is the graph of a linear bijection.

Now consider $a\in X\setminus\{\mathbf{0}\}$ and $b \in Y\setminus \{\mathbf{0}\}$ such that $(a,b)\in R$. Considering the formula $\varphi_{r}(x,\mathbf{0},\mathbf{0})$ for sufficiently large $r$, we have that $$(2-r^{-1})|\log\lVert a \rVert-\log \lVert b \rVert| = (2-r^{-1})\left|\log\frac{\lVert a \rVert}{\lVert b \rVert}\right| \leq \e.$$

Since we can do this for arbitrarily large $r$, this yields

$$2\left|\log\frac{\lVert a \rVert}{\lVert b \rVert}\right| \leq \e.$$

So we have that $\lVert R \rVert \leq e^{\e/2} = \sqrt{e^\e}$ and by symmetry $\lVert R^{-1} \rVert \leq \sqrt{e^\e}$.

$(\Leftarrow):$ Let $A$ be a linear bijection between $X$ and $Y$ such that $\lVert R \rVert, \lVert R^{-1} \rVert \leq \sqrt{e^{\e}}$. Let $R = A \cup \{\infty^\frk{M},\infty^\frk{N}\}$. We need to compute $\mathrm{dis}_{BM}(R) = \mathrm{dis}_{BM_0}(R)$. Since we know that $(\mathbf{0}^\frk{M},\mathbf{0}^\frk{N}) \in R$, we only need to check the formulas in Lemma $\ref{lem:BM-fmlas}$.

Let $(a,e),(b,f),(c,g)\in R$ and consider the quantity $|\varphi_{r}^\frk{M}(a,b,c)-\varphi_{r}^\frk{N}(e,f,g)|$. If any of $a,b,c,e,f,g$ are $\infty$ then this is $0$, so assume that none of them are. To estimate this we will need the following facts:
\begin{align*}
|x_1 y_1 - x_0 y_0| &\leq |x_1 - x_0|(|y_0|\uparrow|y_1|)+|y_1 - y_0|(|x_0|\uparrow|x_1|), \\
|[x]_a^b - [y]_a^b| &\leq |x-y|.
\end{align*}
Applying these to this case gives
\begin{align*}
|\varphi_{r}^\frk{M}(a,b,c)-\varphi_{r}^\frk{N}(e,f,g)| &\leq r^{-2}\left|\log\frac{\lVert a \rVert\uparrow \lVert b\rVert \uparrow \lVert c\rVert}{\lVert e \rVert\uparrow \lVert f\rVert \uparrow \lVert g\rVert}\right|r \\
&+ (2-r^{-1})\left|\log\frac{\lVert a + b - c \rVert}{\lVert e + f - g \rVert}\right|,
\end{align*}
since the first term in $\varphi_r$ can have magnitude at most $1$ and the second term can have magnitude at most $r$. Now finally note that we must have
\begin{align*}
\left|\log\frac{\lVert a \rVert\uparrow \lVert b\rVert \uparrow \lVert c\rVert}{\lVert e \rVert\uparrow \lVert f\rVert \uparrow \lVert g\rVert}\right| &\leq \frac{\e}{2},\\
\left|\log\frac{\lVert a + b - c \rVert}{\lVert e + f - g \rVert}\right| &\leq \frac{\e}{2}.
\end{align*}
Putting this all together gives

$$|\varphi_{r}^\frk{M}(a,b,c)-\varphi_{r}^\frk{N}(e,f,g)| \leq r^{-1}\frac{\e}{2} + (2-r^{-1})\frac{\e}{2} = \e.$$

The same proof works for $\psi_{r,s}$, so we have that $\mathrm{dis}_{\mathrm{BM}}(R) \leq \e$.
\end{proof}

\begin{cor}
If $X$ and $Y$ are Banach spaces and $\frk{M}$ and $\frk{N}$ are their corresponding emboundments, then $\rho_{\mathrm{BM}}(\frk{M},\frk{N})=d_{\mathrm{BM}}(X,Y)$.
\end{cor}

\begin{proof}
Clearly we have $d_{\mathrm{BM}}(X,Y) \leq \rho_{\mathrm{BM}}(\frk{M},\frk{N})$. To get the other direction, let $A:X\rightarrow Y$ be a linear bijection with $\lVert A \rVert \cdot \lVert A^{-1}\rVert \leq e^\e$. If we set $r = \sqrt{\frac{\lVert A^{-1}\rVert}{\lVert A \rVert}}$, then we have that $rA$ is a linear bijection between $X$ and $Y$ with $\lVert rA \rVert \leq \sqrt{e^\e}$ and $\lVert (rA)^{-1} \rVert \leq \sqrt{e^\e}$, so we get $\rho_\mathrm{BM}(\frk{M},\frk{N})\leq d_\mathrm{BM}(X,Y)$.
\end{proof}


\subsection{Approximate Isomorphism in Discrete Logic}

Perhaps surprisingly, the concept of a distortion system is non-trivial in discrete logic.

\begin{defn}
A \emph{stratified language} is a language $\Lcal$ together with a designated sequence of sub-languages $\{\Lcal_i\}_{i<\omega}$ whose union is $\Lcal$. (Note that the sub-languages may have fewer sorts than the full language.)

In the context of a stratified language $\Lcal$, two $\Lcal$-structures $\frk{M}$, $\frk{N}$ are said to be \emph{approximately isomorphic}, written $\frk{M} \approxx_\Lcal \frk{N}$, if $\frk{M}\upharpoonright \Lcal_i \cong \frk{N}\upharpoonright \Lcal_i$ for every $i<\omega$. In general let $\rho_\Lcal(\frk{M},\frk{N}) = 2^{-i}$ where $i$ is the largest such that $\frk{M} \upharpoonright \Lcal_i \cong \frk{N} \upharpoonright \Lcal_i$ but $\frk{M} \upharpoonright \Lcal_{i+1} \not\cong \frk{N} \upharpoonright \Lcal_{i+1}$, or $0$ if no such $i$ exists. 

We may drop the subscript $\Lcal$ if the relevant stratified language is clear by context.
\end{defn}

Clearly $\rho_\Lcal$ is a pseudo-metric on $\Lcal$-structures.

\begin{prop}
Let $T$ be a discrete first-order theory (i.e.\ every predicate is $\{0,1\}$-valued in every model of $T$) and let $\Delta$ be a distortion system for $T$.

\begin{itemize}
    \item[(i)] For every finite set $\mathcal{S}_0\subseteq \mathcal{S}$ there is an $\e > 0$ such that if $\mathrm{dis}_{\Delta}(R) < \e$, then $R$ restricted to the sorts in $\mathcal{S}_0$ is the graph of a bijection. For every predicate symbol $P$ there is an $\e_P > 0$ such that whenever $\mathrm{dis}_{\Delta}(R) < \e_P$, then $R$ is the graph of a bijection that respects $R$.
    \item[(ii)] There is a stratification of $\Lcal$ such that $\rho_\Delta$ and $\rho_\Lcal$ are uniformly equivalent. In particular $\frk{M}\approxx_\Delta\frk{N}$ if and only if $\frk{M} \approxx_\Lcal \frk{N}$. 
\end{itemize}
\end{prop}
\begin{proof}
\emph{(i):} This follows immediately from Lemma \ref{lem:coarsest}.

\emph{(ii):} Choose $\e_P$ as in part \emph{(i)} for all predicate symbols. For each $i<\omega$, let $\mathcal{S}_i$ be the set of sorts such that $\e_{=_s} \geq 2^{-i}$. Set $\Lcal_i$ to be the set of all predicate symbols $P$ such that $\e_P \geq 2^{-i}$ and for every sort $s$ of a variable in $P$, $\e_{=_s}\geq 2^{-i}$. 

Then, for sufficiently small distances, $\rho_\Lcal$ and $\rho_\Delta$ never differ by more than a factor of $4$, so they are uniformly equivalent.
\end{proof}

Note that $\Delta$ for a discrete theory will still contain continuous formulas (since we are implicitly considering it as a continuous theory) and these will be what gives it its structure.

\section{Scott Sentences for Functional Approximation Fragments}
\label{sec:Scott-sent}

Here we will develop back-and-forth pseudo-metrics, $r^\Delta_\alpha$, for arbitrary distortion systems, an extension of \cite{MSA}. In the case of separable structures with functional or u.u.c.\ distortion systems, $r^\Delta_\infty$ will be equal to the corresponding $\rho_\Delta$, but for some irregular distortion systems we will show that $r^\Delta_\infty \neq \rho_\Delta$ (in particular because $r^\Delta_\infty \leq a_\Delta < \rho_\Delta$).

As a corollary of this we will explicitly exhibit Scott sentences for $\Delta$-\hskip0pt equivalence with functional $\Delta$ (which are precisely the same as Ben Yaacov's perturbations \cite{OnPert}). This covers Banach-Mazur equivalence for Banach spaces and Lipschitz equivalence for metric spaces, which were not expressible in the framework of \cite{MSA}, although the existence of these was shown indirectly by the continuous Lopez-Escobar theorem in \cite{MSA} and results in \cite{ComplexDistNew}.

Many of the proofs in this section are nearly identical to the corresponding proofs in \cite{MSA}, so we will only sketch the important parts. We should pause to emphasize that \textbf{for bookkeeping purposes  in this section we are not treating all variables as interchangeable.} For $n<m$ we are thinking of $x_m$ as being (potentially) `more sensitive' than $x_n$, so more formulas are allowed to have $x_m$ as a variable than $x_n$. See section 2 of \cite{MSA}.

\begin{defn} Let $\Delta$ be a collection of formulas closed under renaming variables (typically a distortion system).


For any $\Lcal$-structures $\frk{M},\frk{N}$, $\bar{m}\in\frk{M}$, $\bar{n}\in\frk{N}$, and weak modulus $\Omega$, we define the \emph{$(\Delta,\Omega)$-back-and-forth pseudo-metrics}, $r^{\Delta,\Omega}_\alpha(\frk{M},\bar{m};\frk{N},\bar{n})$ as follows:
\begin{itemize}
    \item $r_0^{\Delta,\Omega}(\frk{M},\bar{m};\frk{N},\bar{n})$ is 
    $$\sup\{|\psi^\frk{M}(\bar{m})-\psi^\frk{N}(\bar{n})|:\psi\in\Delta,\,\psi\text{ respects }\Omega\text{  in every }\Lcal\text{-structure}\}.$$
    \item $r_{\alpha+1}^{\Delta,\Omega}(\frk{M},\bar{m};\frk{N},\bar{n})$ is
    $$\sup_{a\in\frk{M}}\inf_{b\in\frk{N}} r_{\alpha}^{\Delta,\Omega}(\frk{M},\bar{m}a;\frk{N},\bar{n}b)\uparrow \sup_{b\in\frk{N}}\inf_{a\in\frk{M}}r_{\alpha}^{\Delta,\Omega}(\frk{M},\bar{m}a;\frk{N},\bar{n}b).$$
    \item $r_{\lambda}^{\Delta,\Omega}(\frk{M},\bar{m};\frk{N},\bar{n})$ is
    $$\sup_{\alpha < \lambda} r_{\alpha}^{\Delta,\Omega}(\frk{M},\bar{m};\frk{N},\bar{n}),$$ for $\lambda$ a limit or $\infty$.

\end{itemize}

\end{defn}

This is the analog of Lemma 3.2 in \cite{MSA}; the proofs are essentially identical.

\begin{lem}


\begin{itemize}
    \item[(i)] For fixed $\alpha\in \mathrm{Ord}\cup\{\infty\}$ and $k$, $r^{\Delta,\Omega}_\alpha$ is a pseudo-metric on the class of all pairs $(\frk{M},\bar{m})$, with $|\bar{m}|=k$.
    \item[(ii)] For every $\alpha$, $\frk{M}$, and $\bar{a},\bar{b} \in \frk{M}$, $r^{\Delta,\Omega}_\alpha(\frk{M},\bar{a};\frk{M},\bar{b}) \leq d^\Omega(\bar{a},\bar{b})$.
    \item[(iii)] For every $\alpha$, $\frk{M}$, and $\frk{N}$, and $k$, the function $(\bar{m},\bar{n})\mapsto r^{\Delta,\Omega}_\alpha(\frk{M},\bar{m};\frk{N},\bar{n})$ on pairs of $k$-tuples is uniformly continuous on $\frk{M}^k \times \frk{N}^k$ with regards to the max metric.
\end{itemize}
\end{lem}

This is the analog of Lemma 3.3 in \cite{MSA}; again the proofs are essentially identical.

\begin{lem}
\begin{enumerate}[label=(\roman*)]
    \item For every $\alpha \leq \beta$, $r^{\Delta,\Omega}_\alpha \leq r^{\Delta,\Omega}_\beta$.
    \item For every pair of structures $\frk{M},\frk{N}$ with $\dc \frk{M},\dc \frk{N} \leq \kappa$, there is an $\alpha < \kappa^+$ such that $r^{\Delta,\Omega}_\alpha(\frk{M},\bar{m};\frk{N},\bar{n})=r^{\Delta,\Omega}_{\alpha+1}(\frk{M},\bar{m};\frk{N},\bar{n})$ for all pairs of tuples $\bar{m}\in\frk{M}$ and $\bar{n}\in\frk{N}$, which implies that in fact $r^{\Delta,\Omega}_\alpha(\frk{M},\bar{m};\frk{N},\bar{n})=r^{\Delta,\Omega}_{\infty}(\frk{M},\bar{m};\frk{N},\bar{n})$ for all such pairs of tuples.
\end{enumerate}
\end{lem}
 
This is the analog of Proposition 3.4 in \cite{MSA}. See \cite{MSA} for the definition of \emph{shift increasing}.

\begin{prop} \label{prop:tail-witness}
Let $\frk{M},\frk{N}\models T$ be separable. For any $\bar{m}\in\frk{M}$ and $\bar{n}\in\frk{N}$, $r^{\Delta,\Omega}_\infty(\frk{M},\bar{m};\frk{N},\bar{n}) < \e$ if and only if there exists tail-dense sequences $\{a_i\}_{i<\omega} \subseteq \frk{M}$ and $\{b_i\}_{i<\omega} \subseteq \frk{N}$ starting with $\bar{m}$ and $\bar{n}$, respectively, such that

$$\sup_{n<\omega}r^{\Delta,\Omega}_0(\frk{M},a_{<n};\frk{N},b_{<n}) < \e,$$

where a sequence is tail-dense if every final segment of it is metrically dense.
\end{prop}

\begin{cor} \label{cor:nice-b-a-f} Let $\Delta$ be a distortion system for $T$, a theory in a countable language. Let $\Omega$ be a weak modulus.
\begin{itemize}
    \item[(i)] For any models $\frk{M}$ and $\frk{N}$ and tuples $\bar{m}$ and $\bar{n}$, we have that $r^{\Delta,\Omega}_\infty(\frk{M},\bar{m};\allowbreak \frk{N},\bar{n})\leq a_\Delta(\frk{M},\bar{m};\frk{N},\bar{n})$. (In particular since $r^{\Delta,\Omega}_\infty$ is a pseudo-metric, this implies that the function $(\frk{M},\bar{m};\frk{N},\bar{n})\mapsto r^{\Delta,\Omega}_\infty(\frk{M},\bar{m};\frk{N},\bar{n})$ is $2$-Lipschitz in $\rho_\Delta$.)
    \item[(ii)] If $\Delta$ is u.u.c.\ and $\Omega$ is shift increasing with the property that for any $\varphi\in \Delta$, $\varphi(x_0,x_1,x_2,\dots)$ is an $\Omega$-formula, then for any separable models $\frk{M}$ and $\frk{N}$ and tuples $\bar{m}$ and $\bar{n}$, we have that $r^{\Delta,\Omega}_\infty(\frk{M},\bar{m};\frk{N},\bar{n}) = \rho_\Delta(\frk{M},\bar{m};\frk{N},\bar{n})$. 
    \item[(iii)] If $\Delta$ is functional then for any sequence $\{\varphi_i\}_{i<\omega} \subset \Delta$ dense in $\Delta$ in the uniform norm and if $\Omega$ is shift increasing such that for any $i<\omega$ there exists an $n<\omega$ such that $\varphi_i(x_n,x_{n+1},\dots,x_{n+k})$ is an $\Omega$-formula, then there is an $\e > 0$ such that for any separable models $\frk{M}$ and $\frk{N}\models T$, if $r^{\Delta,\Omega}_\infty(\frk{M},\frk{N}) < \e$, then $r^{\Delta,\Omega}_\infty(\frk{M},\frk{N}) = \rho_\Delta(\frk{M},\frk{N})$. 
\end{itemize}
\end{cor}
\begin{proof}
\emph{(i):} If $\frk{M}$ and $\frk{N}$ are separable we can just use the previous proposition.

The idea of the proof is that we can use an almost correlation between $\frk{M}$ and $\frk{N}$ as a back-and-forth strategy. We will proceed by induction, showing that for every $\alpha$, $r^{\Delta,\Omega}_\alpha(\frk{M},\bar{m};\frk{N},\bar{n})\leq a_\Delta(\frk{M},\bar{m};\frk{N},\bar{n})$.

$r^{\Delta,\Omega}_0(\frk{M},\bar{m};\frk{N},\bar{n})\leq a_\Delta(\frk{M},\bar{m};\frk{N},\bar{n})$ clearly holds, as does the limit case.

Assume that $r^{\Delta,\Omega}_\alpha(\frk{M},\bar{m};\frk{N},\bar{n})\leq a_\Delta(\frk{M},\bar{m};\frk{N},\bar{n})$ holds for all tuples $\bar{m}$ and $\bar{n}$ and consider $r^{\Delta,\Omega}_{\alpha+1}(\frk{M},\bar{m};\frk{N},\bar{n})$. Fix an $\e > 0$ and find $R \in \mathrm{acor}(\frk{M},\bar{m};\frk{N},\bar{n})$ such that $\mathrm{dis}_\Delta(R) < a_\Delta(\frk{M},\bar{m};\frk{N},\bar{n}) + \frac{1}{2}\e$.

Now for any $a\in \frk{M}$ find an $a^\prime$ such that there is some $b\in \frk{N}$ with $(a^\prime,b)\in R$ and such that $a$ and $a^\prime$ are close enough that 
\[
|r^{\Delta,\Omega}_{\alpha}(\frk{M},\bar{m}a;\frk{N},\bar{n}b)-r^{\Delta,\Omega}_{\alpha}(\frk{M},\bar{m}a^\prime;\frk{N},\bar{n}b)| \leq \frac{1}{2}\e
\]
 (this exists since $r^{\Delta,\Omega}_{\alpha}$ is uniformly continuous in the tuple arguments). By the induction hypothesis, $r^{\Delta,\Omega}_{\alpha}(\frk{M},\bar{m}a^\prime;\frk{N},\bar{n}b) \leq a_\Delta(\frk{M},\bar{m}a^\prime;\frk{N},\bar{n}b)\leq \mathrm{dis}_\Delta(R)$, so we have that $r^{\Delta,\Omega}_{\alpha}(\frk{M},\bar{m}a;\frk{N},\bar{n}b) < a_\Delta(\frk{M},\bar{m};\frk{N},\bar{n}) +\e $. Since we can do this for any $\e>0$ and we can do the same thing for any $b \in \frk{N}$, we have that $r^{\Delta,\Omega}_{\alpha+1}(\frk{M},\bar{m};\frk{N},\bar{n}) \leq a_\Delta(\frk{M},\bar{m};\frk{N},\bar{n})$.

\emph{(ii):} From part \emph{(i)} we already have that $r^{\Delta,\Omega}_\infty \leq \rho_\Delta$, so we just need to show that $\rho_\Delta \leq r^{\Delta,\Omega}_\infty$.

Fix $\e>0$ and assume that $r^{\Delta,\Omega}_\infty(\frk{M},\bar{m};\frk{N},\bar{n}) < \e$. As guaranteed by Proposition \ref{prop:tail-witness}, let $\bar{a},\bar{b}$ be tail-dense sequences in $\frk{M}$ and $\frk{N}$ which begin with $\bar{m}$ and $\bar{n}$, respectively, such that for every $n<\omega$, $r^{\Delta,\Omega}_0(\frk{M},a_{<n};\frk{N},b_{<n}) < \e$. By the condition on $\Omega$ this implies that $R = \{(a_i,b_i):i<\omega\}$ is an almost correlation between $(\frk{M},\bar{m})$ and $(\frk{N},\bar{n})$ such that $\mathrm{dis}_\Delta(R)\leq \e$. Since $\Delta$ is regular this implies that $\rho_\Delta(\frk{M},\bar{m};\frk{N},\bar{n})\leq \e$. Since we can do this for any $\e > r^{\Delta,\Omega}_\infty(\frk{M},\bar{m};\frk{N},\bar{n})$, we have that $\rho_\Delta(\frk{M},\bar{m};\frk{N},\bar{n})\leq r^{\Delta,\Omega}_\infty(\frk{M},\bar{m};\frk{N},\bar{n})$.

\emph{(iii):} 
From part \emph{(i)} we already have that $r^{\Delta,\Omega}_\infty \leq \rho_\Delta$, so we just need to show that $\rho_\Delta \leq r^{\Delta,\Omega}_\infty$.

By Proposition \ref{prop:nice-regular} part \emph{(ii)}, there is a $\delta > 0$ such that for any $\gamma > 0$ there is a formula $\psi(x,y) \in \Delta$ such that for any $\frk{M}\models T$ and $a,b\in \frk{M}$, $\psi^\frk{M}(a,a)=0$, and if $\psi^\frk{M}(a,b) < \delta$, then $d^\frk{M}(a,b) < \gamma$. By the density of the sequence $\{\varphi_i\}_{i<\omega}$, there is an $m(\gamma)$ such that $\lVert \varphi_{m(\gamma)} -\psi \rVert_\infty < \frac{1}{3}\delta$ (in any $\Lcal$-structure). This implies that for any $\frk{M}\models T$ and $a,b\in\frk{M}$, if $\varphi_{m(\gamma)}^\frk{M}(a,a) < \frac{1}{3}\delta$ and $\varphi_{m(\gamma)}^\frk{M}(a,b)<\frac{2}{3}\delta$, then $d^\frk{M}(a,b) < \gamma$.

Now  assume that $r^{\Delta,\Omega}_\infty(\frk{M},\frk{N}) < \frac{1}{6}\delta$ and pick $\eta$ such that $r^{\Delta,\Omega}_\infty(\frk{M},\frk{N}) < \eta < \frac{1}{6}\delta$. As guaranteed by Proposition \ref{prop:tail-witness}, let $\bar{a},\bar{b}$ be tail-dense sequences in $\frk{M}$ and $\frk{N}$, respectively, such that for every $n<\omega$, $r^{\Delta,\Omega}_0(\frk{M},a_{<n};\frk{N},b_{<n}) < \eta$.

Let $R$ be the set of all pairs $(c,e)\in\frk{M}\times \frk{N}$ such that there exists a sequence $\{i(j)\}_{j<\omega}$ of natural numbers such that $\lim_{j\rightarrow \infty} i(j) =\infty$ and $\{a_{i(j)}\}_{j<\omega}$ is a Cauchy sequence limiting to $c$ and $\{b_{i(j)}\}_{j<\omega}$ is a Cauchy sequence limiting to $e$. By the uniform continuity of formulas and the fact that $\Omega$ is shift-increasing, it is clear that $\mathrm{dis}_{\{\varphi_i\}}(R) \leq \eta$. Since $\{\varphi_i\}$ is dense in $\Delta$ this implies that $\mathrm{dis}_{\Delta}(R) \leq \eta$. 

So now we just need to show that $R$ is a correlation. Pick $c\in\frk{M}$ and let $\{i(j)\}_{j<\omega}$ be a sequence of natural numbers such that $\lim_{j\rightarrow \infty} i(j) =\infty$ and $\{a_{i(j)}\}_{j<\omega}$ is a Cauchy sequence limiting to $c$, which must exist by the tail-denseness of $\{a_i\}_{i<\omega}$. Consider the sequence $\{b_{i(j)}\}_{j<\omega}$.

Pick $\gamma > 0$ and consider the formula $\varphi_{m(\gamma)}$, as specified above. Find a $\sigma > 0$ such that if $d(xy,zw) < \sigma$, then $|\varphi_{m(\gamma)}(x,y) - \varphi_{m(\gamma)}(z,w)| < \frac{1}{6}\delta$ (in any $\Lcal$-structure). 

Find an $N(\gamma)$ such that $\varphi_{m(\gamma)}(x_{N(\gamma)},x_{N(\gamma)+1})$ is an $\Omega$-formula and such that for all $j,k\geq N(\gamma)$, $d^\frk{M}(a_{i(j)},a_{i(k)}) < \sigma$. This implies that for any $j,k \geq N(\gamma)$, $$|\varphi_{m(\gamma)}^\frk{M}(a_{i(j)},a_{i(k)}) - \varphi_{m(\gamma)}^\frk{M}(a_{i(j)},a_{i(j)})| < \frac{1}{6}\delta,$$ so in particular $$\varphi_{m(\gamma)}^\frk{M}(a_{i(j)},a_{i(k)}) < \frac{1}{3}\delta + \frac{1}{6}\delta =\frac{1}{2}\delta < \frac{2}{3}\delta.$$

This implies that for any $j,k\geq N(\gamma)$, $$\varphi_{m(\gamma)}^\frk{N}(b_{i(j)},b_{i(k)}) < \frac{1}{2}\delta + \eta < \frac{1}{2}\delta + \frac{1}{6}\delta = \frac{2}{3}\delta.$$
By construction this implies that $d^\frk{N}(b_{i(j)},b_{i(k)}) < \gamma$. Since we can do this for any $\gamma > 0$, we have that $\{b_{i(j)}\}_{j<\omega}$ is a Cauchy sequence in $\frk{N}$, converging to some point $e$, so we have that $(c,e)\in R$.

By symmetry we can do the same for Cauchy sequences in $\frk{N}$, showing that $R$ is a correlation, so we have that $\rho_\Delta(\frk{M},\frk{N}) \leq r^{\Delta,\Omega}_\infty(\frk{M},\frk{N})$. Therefore $ r^{\Delta,\Omega}_\infty(\frk{M},\frk{N}) = \rho_\Delta(\frk{M},\frk{N})$ whenever $r^{\Delta,\Omega}_\infty(\frk{M},\frk{N}) <  \e = \frac{1}{6}\delta$.
\end{proof}

We should note that the case of u.u.c.\ distortion systems is very close to something that can be captured by the original formalism in \cite{MSA}. In particular if $\Delta$ is a u.u.c.\ distortion system for a first-order theory $T$, then $T$ is interdefinable with a theory $T^\prime$ in a uniformly Lipschitz language \cite{MetSpaUniv} and the back-and-forth pseudo-metric coming from the $1$-Lipschitz weak modulus will be uniformly equivalent to the original $\rho_\Delta$ for separable structures. 

In cases where we know that $a_\Delta$ is not a pseudo-metric, we automatically know from part \emph{(i)} that $r^{\Delta,\Omega}_\infty < a_\Delta < \rho_\Delta$, since $r^{\Delta,\Omega}_\infty$ and $\rho_\Delta$ are pseudo-metrics.

So now we can continue on to construct Scott sentences. This the analog of Definition 3.6 in \cite{MSA}.

\begin{defn}
Let $\Delta$ be a collection of formulas closed under renaming variables and $\Omega$ a weak modulus. For a pair of models $\frk{M},\frk{N}\models T$, $\alpha_{\frk{M},\frk{N}}$ is the least ordinal $\alpha$ such that for all tuples $\bar{m}\in\frk{M}$ and $\bar{n}\in\frk{N}$, $r^{\Delta,\Omega}_{\alpha}(\frk{M},\bar{m};\frk{N},\bar{n}) = r^{\Delta,\Omega}_{\alpha+1}(\frk{M},\bar{m};\frk{N},\bar{n})$. This is called the \emph{$(\Delta,\Omega)$-Scott rank of the pair $\frk{M}$ and $\frk{N}$}. If $\frk{M}=\frk{N}$ we just write $\alpha_\frk{M}$, which is the \emph{$(\Delta,\Omega)$-Scott rank of $\frk{M}$}.
\end{defn}

Just like Lemma 3.7 in \cite{MSA} we have that if $r^{\Delta,\Omega}_\infty(\frk{M},\frk{N})=0$, then $\alpha_{\frk{M}}=\alpha_{\frk{M},\frk{N}}=\alpha_{\frk{N}}$.

To define Scott sentences we need to specify what we mean by $\Lcal_{\omega_1\omega}^{\Delta,\Omega}$. 

\begin{defn}
Given a collection of first-order formulas $\Delta$, closed under renaming variables, and a weak modulus $\Omega$,  \emph{$n$-ary $\Lcal_{\omega_1\omega}^{\Delta,\Omega}$-formulas} are defined inductively. We also need to inductively define the \emph{syntactic range}, written $I(\varphi)$, of such formulas. For first-order formulas this is the convex closure of the set of possible values of the formula in $\Lcal$-structures, which is always a compact interval.
\begin{itemize}
    \item If $\varphi \in \Delta$ and $\varphi(x_0,\dots,x_{n-1})$ obeys $\Omega$, then $\varphi(x_0,\dots,x_{n-1})$ is an $n$-ary $\Lcal_{\omega_1\omega}^{\Delta,\Omega}$-formula. $I(\varphi)$ is the syntactic range of $\varphi$ as a first-order formula.

    \item For any compact interval $I$, if $\{\varphi_i\}_{i<\omega}$ is a sequence of $n$-ary $\Lcal_{\omega_1\omega}^{\Delta,\Omega}$-formulas such that $I$ is the closure of $\bigcup_{i<\omega}I(\varphi_i)$, then $\psi = \sup_i \varphi_i$ and $\chi = \inf_i \varphi_i$ are $n$-ary $\Lcal_{\omega_1\omega}^{\Delta,\Omega}$-formulas and $I(\psi)=I(\chi)=I$.
    \item If $\varphi$ is an $(n+1)$-ary  $\Lcal_{\omega_1\omega}^{\Delta,\Omega}$-formula then $\psi = \sup_{x_n}\varphi$ and $\chi=\inf_{x_n}\varphi$ are $n$-ary $\Lcal_{\omega_1\omega}^{\Delta,\Omega}$-formulas and $I(\psi)=I(\chi)=I$.    
    \item If $\varphi_1,\dots,\varphi_k$ is a finite list of $n$-ary $\Lcal_{\omega_1\omega}^{\Delta,\Omega}$-formulas and $F:\mathbb{R}^k \rightarrow \mathbb{R}$ is a $1$-Lipschitz connective, then $\psi = F(\varphi_1,\dots,\varphi_k)$ is an $n$-ary $\Lcal_{\omega_1\omega}^{\Delta,\Omega}$-formula and $I(\psi)$ is the image of $I(\varphi_1)\times\dots\times I(\varphi_k)$ under $F$ (which is always a compact interval).

An \emph{$\Lcal_{\omega_1\omega}^{\Delta,\Omega}$-formula} is an $n$-ary $\Lcal_{\omega_1\omega}^{\Delta,\Omega}$-formula for some $n$ and an $\Lcal_{\omega_1\omega}^{\Delta,\Omega}$-sentence is a $0$-ary $\Lcal_{\omega_1\omega}^{\Delta,\Omega}$-formula. \qedhere
\end{itemize}
\end{defn}

The interpretation of an $\Lcal_{\omega_1\omega}^{\Delta,\Omega}$-formula in a $\Lcal$-structure is obvious. It is also not hard to show that an $n$-ary $\Lcal_{\omega_1\omega}^{\Delta,\Omega}$-formula $\varphi$ is always $\Omega{\upharpoonright} n$-uniformly continuous and can only take on values in the interval $I(\varphi)$ in $\Lcal$-structures.

Next, just like in \cite{MSA} we get that for every countable ordinal $\alpha$, $n<\omega$, separable model $\frk{M}\models T$, and tuple $\bar{m}\in\frk{M}$, there is an $\Lcal_{\omega_1\omega}^{\Delta,\Omega}$-formula $\varphi_{\alpha, \frk{M},\bar{m}}$ such that for all models $\frk{N}\models T$ and tuples $\bar{n}\in\frk{N}$,
$$\varphi_{\alpha,n, \frk{M},\bar{m}}^\frk{N}(\bar{n})=r^{\Delta,\Omega}_\alpha(\frk{M},\bar{m};\frk{N},\bar{n}) \downarrow 1.$$
Fix a countable dense sub-pre-structure $\frk{M}_0=\{a_i\}_{i<\omega}$. We define these formulas inductively. 
$$\varphi_{0,n, \frk{M},\bar{m}}(x_0,\dots,x_{n-1}) = \sup_i |\psi_i^\frk{M}(\bar{a})-\psi_i(\bar{x})|,$$
where $\{\psi_i\}_{i<\omega}$ is a countable sequence of $n$-ary $\Delta$-formulas respecting $\Omega$ that are dense in the uniform norm in the collection of $n$-ary $\Delta$-formulas respecting $\Omega$. 

For a successor stage  we define $\varphi_{\alpha+1,n, \frk{M},\bar{m}}(x_0,\dots,x_{n-1})$ to be
$$\sup_i \inf_{x_n} \varphi_{\alpha,n+1, \frk{M},\bar{m}a_i}(x_0,\dots,x_n)\uparrow \sup_{x_n}\inf_i\varphi_{\alpha,n+1, \frk{M},\bar{m}a_i}(x_0,\dots,x_n).$$
And then for limit $\lambda$, obviously we define
$$\varphi_{\lambda, n, \frk{M},\bar{m}}(x_0,\dots,x_{n-1}) = \sup_{\alpha < \lambda}\varphi_{\alpha,n, \frk{M},\bar{m}}(x_0,\dots,x_{n-1}).$$

Now finally if $\alpha = \alpha_\frk{M}$ is the $(\Delta,\Omega)$-Scott rank of $\frk{M}$, then we define the $(\Delta,\Omega)$-Scott sentence, $\sigma_\frk{M}$, by
$$\sigma_\frk{M} = \varphi_{\alpha,0,\frk{M}}\uparrow \sup_{n<\omega,\bar{a}\in\frk{M}_0^n} \sup_{x_0,\dots,x_{n-1}} \frac{1}{2}|\varphi_{\alpha,n, \frk{M},\bar{a}}-\varphi_{\alpha+1,n, \frk{M},\bar{a}}|, $$
i.e.\ $\frk{N}\models \sigma_\frk{M} \leq 0$ if and only if $\alpha_\frk{N}=\alpha_\frk{M}$ and $r^{\Delta,\Omega}_\alpha(\frk{M},\frk{N})=0$.

Now we get the following analog of Theorem 3.8 in \cite{MSA}.

\begin{thm} \label{thm:Scott-sent}
If $\Delta$ is a u.u.c.\ or functional distortion system for $T$, then for any separable $\frk{M},\frk{N}\models T$, $\frk{N}\models \sigma_\frk{M} \leq 0$ if and only if $\rho_\Delta(\frk{M},\frk{N})=0$.
\end{thm}

Given the comment after Corollary \ref{cor:nice-b-a-f}, we know that Theorem \ref{thm:Scott-sent} is simply false for irregular distortion systems (and we will examine an example in Section \ref{sec:Irreg}), so two natural questions arise:

\begin{quest}
Do Corollary \ref{cor:nice-b-a-f} and Theorem \ref{thm:Scott-sent} hold for any regular distortion system with the appropriate choice of $\Omega$?
\end{quest}

\begin{quest}
For a fixed separable model $\frk{M}\models T$, are either of the collections of separable models $\{\frk{N} : \rho_\Delta(\frk{M},\frk{N})=0\}$ or $\{\frk{N} : a_\Delta(\frk{M},\frk{N})=0\}$ Borel in the sense of \cite{MSA} when $\Delta$ is an irregular distortion system?
\end{quest}

\section{Stability of Bounded $\Rb$-trees and ultrametric spaces}

Let $\Lc_{r}$ be the metric signature with a single constant symbol $p$ and a $[0,2r]$-valued metric. In \cite{Carlisle2020}, it is established that there is an $\Lc_r$-theory $\RTr$ whose models are precisely the (metrically complete) pointed $\Rb$-trees of radius at most $r$. Furthermore, \cite{Carlisle2020} shows that $\RTr$ has a model companion $\mathrm{rb}\RTr$ (for richly branching $\Rb$-tree of radius $r$) and that $\mathrm{rb}\RTr$ is strictly stable. Here we will generalize this result, and then use this generalization to show that all ultrametric spaces are stable.

\begin{prop}\label{prop:stab-delta-closed}
  Fix an incomplete $\Lc$-theory $T$ and a distortion system $\Delta$ for $T$. The set $\{T' \in S_0(T) : T'~\text{is stable}\}$ is closed in the metric $\delta_\Delta^0$.
\end{prop}
\begin{proof}
  Fix an unstable completion $T_0 \supset T$. Let $(a_i)_{i<\omega}$ be a sequence of elements of some $\frk{M} \models T_0$ such that for some formula $\varphi(x,y)$ and $r<s$, $\varphi(a_i,a_j)\leq r$ if $i\leq j$ and $\varphi(a_i,a_j) \geq s$ if $j < i$. Since $\Delta$ is logically complete, we may assume that $\varphi(x,y)$ is in $\Delta$ (changing the values of $r$ and $s$ if necessary).

  Fix $\e > 0$ with $\e < \frac{1}{3}(s-r)$. Suppose that $T_1 \supset T$ is a completion such that $\delta_\Delta^0(T_0,T_1) < \e$. By passing to a larger $\frk{M}$ if necessary, we may assume that $\rho_\Delta(\frk{M},\frk{N}) < \e$ for some $\frk{N} \models T_1$. Let $R \subseteq \frk{M} \times \frk{N}$ be a correlation such that $\mathrm{dis}_\Delta(R) < \e$. Let $(b_i)_{i<\omega}$ be a sequence such that $R(a_i,b_i)$ for all $i<\omega$. We now immediately have that $\varphi(b_i,b_j)< \frac{2}{3}r + \frac{1}{3}s$ if $i\leq j$ and $\varphi(b_i,b_j) > \frac{1}{3}r+\frac{2}{3}s$ if $j < i$. Therefore, the sequence $(b_i)_{i<\omega}$ witnesses that $T_1$ is unstable as well.

  Since we can do this for any unstable completion of $T$, we have that the set of unstable completions of $T$ is $\delta^0_\Delta$-open, so we are done.
\end{proof}

Note that Proposition~\ref{prop:stab-delta-closed} easily generalizes to many dividing lines, such as simplicity, NIP, NTP$_2$, NSOP$_1$, etc.

\begin{prop}\label{prop:R-tree-stab}
  For any $r \geq 0$, every completion of $\RTr$ is stable.
\end{prop}
\begin{proof}
  Fix a metrically complete rooted $\Rb$-tree $E$ of radius at most $r$. Let $\Delta$ be the distortion system generated by the formulas $\frac{1}{2}d(x,y)$ and $\frac{1}{2}d(x,p)$. It is immediate that $\Delta$ is atomically complete and therefore that $\overline{\Delta}$ is a distortion system by Proposition~\ref{prop:atom-comp}. ($\overline{\Delta}$ is a pointed version of the Gromov-Hausdorff distortion system.) For any positive $n<\omega$, let $E_n$ be the sub-structure of $E$ consiting of those $a \in E$ such that $\frac{n}{r}d(x,p)$ is a natural number. Let $R_n \subseteq E \times E_n$ be the relation $d(x,y) \leq \frac{r}{n}$.

  Proposition~\ref{prop:enlarge} and a direct calculation show that for any positive $n<\omega$,
  \[
\mathrm{dis}_\Delta R_n = \mathrm{dis}_{\overline{\Delta}}R_n \leq \frac{1}{2}\left(2\frac{r}{n}\right) = \frac{r}{n},
\]
whence $\mathrm{Th}(E_n)\to \mathrm{Th}(E)$ in the metric $\delta^0_\Delta$.

Finally, it is a classical result that the theory of any rooted tree of finite height is superstable. Each of the structures $E_n$ is bi-interpretable with a discrete rooted tree of height at most $n$, so by Proposition~\ref{prop:stab-delta-closed}, we have that $\mathrm{Th}(E)$ is stable, as required.
\end{proof}

Note, however, that we cannot conclude that completions of $\RTr$ are always \emph{super}stable, as demonstrated in \cite{Carlisle2020}.

\begin{cor}
  For any ultrametric space $(X,d)$ with finite diameter, $\mathrm{Th}(X)$ (in the empty language) is stable.
\end{cor}
\begin{proof}
  Let the diameter of $X$ be $r$. Put a pseudo-metric on $X \times [0,r]$ by
  \[
\rho((x,s),(y,t)) = \inf\{(u-s)+(u-t) : (\exists z \in X) B_{\leq s}(x) \cup B_{\leq t}(y)\subseteq B_{\leq u}(z)\}.
\]
The following facts are easy to verify: For any $x,y \in X$, $\rho((x,r),(y,r)) = 0$. $(X\times [0,r])/\rho$ is an $\Rb$-tree. The map $x \mapsto (x,0)$ is an isometric embedding of $X$ into $(X\times [0,r])/\rho$.

Taking $p$ to be $(x,r)$ (for any $x \in X$) gives a pointed $\Rb$-tree of radius $r$. The image of $X$ under the isometric embedding $(X \times [0,r])/\rho$ is definable by the formula $r-d(x,p)$. By Proposition~\ref{prop:R-tree-stab}, $\mathrm{Th}((X\times [0,r])/\rho, p)$ is stable. Therefore $\mathrm{Th}(X)$ is stable as well, since $X$ is isometric to a definable subset of a stable structure.
\end{proof}

\section{An Irregular Distortion System} \label{sec:Irreg}

Here we will give some explicit examples of the pathological behavior of $a_\Delta$ and the separation of $r^{\Delta,\Omega}_\infty$ and $\rho_\Delta$ for irregular $\Delta$.

\begin{defn}
Let $\Lcal$ be the single-sorted language with a $[0,1]$-valued metric and a single $[0,1]$-valued $1$-Lipschitz unary predicate $U$.

Let $\mathrm{GH}_0 =\{\frac{1}{2}d(x,y)\}$ and let $\mathrm{GH}=\ol{\mathrm{GH}_0}$.

Let $\mathrm{IU}_0 = \mathrm{GH}_0\cup\{nU(x)\}_{n<\omega}$ and let $\mathrm{IU} = \ol{\mathrm{IU}_0}$.
\end{defn}

Clearly $\mathrm{GH}$ corresponds to the ordinary Gromov-Hausdorff distance, ignoring $U$. The notion of approximate isomorphism induced by $\mathrm{IU}$ is strange.

\begin{prop} \label{prop:irreg-ex1}
Let $\frk{M},\frk{N}$ be $\Lcal$-structures and let $R \subseteq \frk{M}\times \frk{N}$. 

$\mathrm{dis}_\mathrm{IU}(R) \leq \e <\infty$ if and only if $\mathrm{dis}_{\mathrm{GH}}(R) \leq \e$ and for all $(a,b) \in R$, $U^\frk{M}(a)=U^\frk{N}(b)$.
\end{prop}

So $R$ needs to be a correlation that rigidly obeys $U$ but which is as loose as the Gromov-Hausdorff distance for $d$.

\begin{defn}
For any $D \subseteq [0,1]$ and $\e \in [0,1]$, let $\frk{I}(D,\e)$ be the $\Lcal$-pre-structure whose universe is $D$ with $U^{\frk{I}(D,\e)}(x)=x$ for $x \in D$ and $d^{\frk{I}(D,\e)}(x,y) = |x-y|\uparrow \e$ for $x,y\in D$ with $x\neq y$.
\end{defn}

\begin{prop}
$(\mathrm{Mod}(T,\leq \omega),\rho_{\mathrm{IU}})$ is not complete as a metric space, where $T$ is the empty theory in the language $\Lcal$.
\end{prop}

\begin{proof}
Let $X = \{2^{-i}\}_{i<\omega}$ and consider the sequence of structures $\{\frk{I}(X,2^{-k})\}_{k<\omega}$. These converge in $\rho_{\mathrm{IU}}$ to the \emph{pre-structure} $\frk{I}(X,0)$, but the corresponding completion is $\frk{I}(X\cup\{0\},0)$ and $\rho_\Delta(\frk{I}(X,0),\frk{I}(X\cup\{0\},0))=\infty$. 
\end{proof}


\begin{prop}
Let $D_0$ and $D_1$ be countable, disjoint, dense subsets of $[0,1]$, then
\begin{align*}
 \rho_\mathrm{IU}(\frk{I}(D_0,\e),\frk{I}([0,1],0)) &= \infty, \\
 \rho_\mathrm{IU}(\frk{I}(D_1,\e),\frk{I}([0,1],0)) &= \infty, \\
 \rho_\mathrm{IU}(\frk{I}(D_0,\e),\frk{I}(D_1,\e)) &= \infty,
 \end{align*}
but
\begin{align*}
a_\mathrm{IU}(\frk{I}(D_0,\e),\frk{I}([0,1],0)) &= \frac{1}{2}\e, \\
a_\mathrm{IU}(\frk{I}(D_1,\e),\frk{I}([0,1],0)) &= \frac{1}{2}\e, \\
a_\mathrm{IU}(\frk{I}(D_0,\e),\frk{I}(D_1,\e)) &= \infty.
\end{align*}

Furthermore $\frk{I}(D_0,\e)$, $\frk{I}(D_1,\e)$, and $\frk{I}([0,1],0)$ are all metrically complete and separable.
\end{prop}
\begin{proof}
For $\rho_\Delta$ there simply are no correlations between these structures that satisfy the requirement on $U$ given in Proposition \ref{prop:irreg-ex1}. Any almost correlation between $\frk{I}(D_0,\e)$ and $\frk{I}(D_1,\e)$ is automatically a correlation, so the same holds for almost correlations.

For $\frk{I}(D_i,\e)$ and $\frk{I}([0,1],\e)$, let $R_i = \{(x,x):x\in D_i\}$, then we have that $\mathrm{dis}_\Delta(R) \leq \frac{1}{2}\e$. Finally  $R_i$ is clearly an almost correlation so we have that $a_\Delta(\frk{I}(D_i,\e),\allowbreak\frk{I}([0,1],0)) \leq \frac{1}{2}\e$. To see that they are actually equal to $\frac{1}{2}\e$, note that these are the \emph{only} almost correlations between these structures with finite $\mathrm{IU}$-distortion.
\end{proof}


We know that for any given weak modulus $\Omega$ something different must happen with $r^{\mathrm{IU},\Omega}_\infty$ on these structures, because $r^{\mathrm{IU},\Omega}_\infty$ is a pseudo-metric, so it cannot be equal to $a_{\mathrm{IU}}$, and therefore also cannot be equal to $\rho_{\mathrm{IU}}$. In particular we must have $r^{\mathrm{IU},\Omega}_\infty(\frk{I}(D_0,\e),\frk{I}(D_1,\e)) \leq \e$.

Finally to get an example of $\rho_\Delta(\frk{M},\frk{N})=\infty$ yet $a_\Delta(\frk{M},\frk{N})=0$ for $\frk{M}$ and $\frk{N}$ complete structures, fix a countable dense $D\subseteq [0,1]$ and let
\begin{align*}
\frk{M} &= \bigsqcup_{i<\omega} \frk{I}(D,2^{-i}),\\
\frk{N} &= \frk{I}([0,1],0)\sqcup \frk{M},
\end{align*}
where in a disjoint union the distances between things in different structures are always $1$. It is easy to check that $\frk{M}$ and $\frk{N}$ are separable $\Lcal$-structures.

\begin{prop}
$\rho_\mathrm{IU}(\frk{M},\frk{N})=\infty$ but $a_\mathrm{IU}(\frk{M},\frk{N}) = 0$.
\end{prop}
\begin{proof}
$\rho_\mathrm{IU}(\frk{M},\frk{N})=\infty$ is clear because there are no correlations between $\frk{M}$ and $\frk{N}$ which can correlate the elements of $[0,1]\setminus D$ in $\frk{N}$ to anything in $\frk{M}$ while having finite distortion.

To see that $a_\mathrm{IU}(\frk{M},\frk{N})=0$, for each $k<\omega$, let $R_k\subseteq \frk{M}\times\frk{N}$ be the almost correlation that relates the copies of $\frk{I}(D,2^{-i})$ for $i < k$ isomorphically, which relates $\frk{I}(D,2^{-k}) \subset \frk{M}$ to the subset of $\frk{I}([0,1],0) \subset \frk{N}$ whose points are those in $D$, and which relates $\frk{I}(D,2^{-i-1}) \subset \frk{M}$ to $\frk{I}(D,2^{-i} \subset \frk{N})$ for $i > k$. Then we have that $\mathrm{dis}_\mathrm{IU}(R_k)=\mathrm{dis}_{\mathrm{IU}_0}(R_k)\leq 2^{-k-1}$, so $a_\mathrm{IU}(\frk{M},\frk{N})\leq 2^{-k-1}$ for each $k<\omega$ and $a_\mathrm{IU}(\frk{M},\frk{N}) = 0$.
\end{proof}

This all raises the question of whether or not $a_\Delta = 0$ is an equivalence relation for irregular $\Delta$. In general we have that for any $\Omega$, that $a_\Delta(\frk{M},\frk{N})=0$ implies $r^{\Delta,\Omega}_\infty(\frk{M},\frk{N})=0$, which is an equivalence relation.

\begin{quest}
Is the relation $a_\Delta(\frk{M},\frk{N})=0$ transitive for irregular $\Delta$?
\end{quest}

Given the fact that $a_\Delta$ is not always a pseudo-metric, it would be very surprising if the answer were yes.



\bibliographystyle{plain}
\bibliography{../ref}

\end{document}